\documentclass[a4paper, reqno, 12pt]{amsart}

\usepackage[utf8]{inputenc}
\usepackage{amsthm,amsfonts,amssymb,amsmath,amsxtra,amsrefs}
\usepackage{mathtools}
\usepackage{bm}
\usepackage{latexsym}
\usepackage{stmaryrd}
\usepackage{graphicx}
\usepackage{tikz}
\usepackage{tikz-cd}
\usepackage[margin=1.25in]{geometry}
\usepackage{mathrsfs}
\usepackage{bm}
\usepackage{comment}

\usepackage[all]{xy}
\SelectTips{cm}{}
\usepackage{xr-hyper}
\usepackage[colorlinks=false,
   citecolor=Black,
   linkcolor=Red,
   urlcolor=Blue]{hyperref}
\usepackage{verbatim}

\RequirePackage{xspace}
\RequirePackage{etoolbox}
\RequirePackage{varwidth}
\RequirePackage{enumitem}
\RequirePackage{tensor}
\RequirePackage{mathtools}
\RequirePackage{longtable}
\RequirePackage{multirow}

\usepackage{rotating}

\tikzset{
    labl/.style={anchor=south, rotate=90, inner sep=.5mm}
}

\tikzset{>=latex}
\tikzstyle{vthick}=[line width=1.8pt]

\newcommand\drawpath[2]{%
  \foreach \too [count=\c from 1] in {#1}
  {
  \ifthenelse{\c=1}
  {\xdef\from{\too}}
  {\path (\from) edge [->, #2] (\too);
    \xdef\from{\too}}
  };
}

\newtheorem{thm}{Theorem}
\newtheorem{theorem}[thm]{Theorem}
\newtheorem*{thmintro}{Theorem}

\newtheorem{prop}[thm]{Proposition}

\newtheorem{lem}[thm]{Lemma}

\newtheorem{cor}[thm]{Corollary}

\theoremstyle{definition}

\newtheorem{example}[thm]{Example}

\newtheorem{remark}[thm]{Remark}

\numberwithin{equation}{section}
\numberwithin{thm}{section}

\newcommand{\red}[1]{{\color{red}#1}}

\newcommand{\BC}{\ensuremath{\mathbb {C}}\xspace}

\newcommand{{\BG}}{\ensuremath{\mathbb {G}}\xspace}

\newcommand{{\BK}}{\ensuremath{\mathbb {K}}\xspace}

\newcommand{\BQ}{\ensuremath{\mathbb {Q}}\xspace}

\newcommand{\BT}{\ensuremath{\mathbb {T}}\xspace}

\newcommand{\BZ}{\ensuremath{\mathbb {Z}}\xspace}

\newcommand{\CA}{\ensuremath{\mathcal {A}}\xspace}
\newcommand{\CB}{\ensuremath{\mathcal {B}}\xspace}

\newcommand{\CI}{\ensuremath{\mathcal {I}}\xspace}

\newcommand{\CL}{\ensuremath{\mathcal {L}}\xspace}

\newcommand{\CO}{\ensuremath{\mathcal {O}}\xspace}
\newcommand{\CP}{\ensuremath{\mathcal {P}}\xspace}

\newcommand{\CU}{\ensuremath{\mathcal {U}}\xspace}

\newcommand{\fg}{\ensuremath{\mathfrak {g}}\xspace}
\newcommand{\fh}{\ensuremath{\mathfrak {h}}\xspace}

\newcommand{\fp}{\ensuremath{\mathfrak {p}}\xspace}
\newcommand{\fq}{\ensuremath{\mathfrak {q}}\xspace}

\newcommand{\x}{\times}

\newcommand{\bfi}{\textnormal{\textbf{i}}}
\newcommand{\bfj}{\textnormal{\textbf{j}}}

\newcommand{\bfs}{\textnormal{\textbf{s}}}

\newcommand{\bfw}{\textnormal{\textbf{w}}}

\newcommand{\uf}{\textnormal{uf}}
\newcommand{\fro}{\textnormal{fr}}
\newcommand{\sgn}{\textnormal{sgn}}

\newcommand{\ve}{\varepsilon}
\newcommand{\isoto}{\xrightarrow{\sim}}
\newcommand{\rank}{\textnormal{rank}}
\newcommand{\LP}{\CL\CP}
\newcommand{\Spec}{\textnormal{Spec }}
\newcommand{\wt}{\widetilde}
\newcommand{\Sym}{\textnormal{Sym}}
\newcommand{\st}{\textnormal{st}}
\newcommand{\Ad}{\textnormal{Ad}}

\begin{document}

\title[]{Cluster structures on $SL_n/SO_n$}

\author[Jeff York Ye]{Jeff York Ye}
\address{Department of Mathematics, National University of Singapore, Singapore.}
\email{e1124873@u.nus.edu}
\subjclass[2020]{} 

\begin{abstract}
We construct cluster structures on the strata $\mathring{S}_w$ of a stratification of the symmetric space $SL_n/SO_n$ over $\BC$. To accomplish this, we study foldings of upper cluster algebras and show that the cluster structures on $SL_n/SO_n$ can be obtained from those on $SL_n$ via folding. As a corollary, we show that these cluster structures are compatible with the De Concini Poisson structures, and we construct cluster structures on the variety of symmetric matrices $\Sym_n$.
\end{abstract}

\maketitle
	
\tableofcontents
 
\section{Introduction}
\label{sec:intro}

\subsection{Cluster algebras}
\label{sec:intro:cluster}

The theory of cluster algebras was developed by Fomin
and Zelevinsky \cite{FZ02} to study total positivity and dual canonical bases in semisimple
groups. As an important example, let $G$ be a simply connected, connected simple algebraic group over $\BC$. The coordinate ring of $G$, as well as the coordinate ring of any double Bruhat cell $G^{u,v}$, is a cluster algebra \citelist{\cite{BFZ05}\cite{QY25}\cite{Oya25}}.

The existence of a cluster structure on a variety has a wide range of consequences. For example, it provides a total positivity structure, Poisson structure \cite{GSV10}, and a cluster theoretic canonical basis on its coordinate ring \cite{GHKK18}.

Foldings for general cluster algebras were introduced in \citelist{\cite{Dem08}\cite{Dup08}}, as analogs of the foldings for Kac--Moody Lie algebras and quantum groups \citelist{\cite{Kac83}\cite{Lus93}}. Since the requirements for folding of cluster algebras are strict, it has been applied only in limited scenarios, for example \citelist{\cite{FG06}\cite{FST12}\cite{IM24}\cite{Zho20}}.

\subsection{Symmetric spaces}
\label{sec:intro:symspace}

Let $G$ be a connected simple algebraic group over $\BC$, $\theta$ be an involution of $G$, and $K$ be the fixed-point subgroup. The quotient $G/K$ is called a symmetric space, and it is isomorphic to a closed subvariety $S$ of $G$ via the embedding $gK\mapsto g\theta(g)^{-1}$. $S$ is a twisted conjugacy class of $G$, hence a Poisson submanifold of $(G,\pi_\theta)$ for a Poisson structure $\pi_\theta$ on $G$ constructed in \cite{Lu14}. The induced Poisson structure on $G/K$ coincides with De Concini's Poisson structure studied in \cite{EE05}*{Section~6.4}.

Let $\theta$ be quasi-split, i.e. we choose a pair of opposite Borel subgroups $B^\pm$ such that $\theta(B^+)=B^-$. Then $\theta$ induces an involution on the indices $I$ and the Weyl group $W$. For any $w\in W$, let $\mathring{S}_w$ be the intersection of $S$ with the double Bruhat cell $G^{w,\theta(w)^{-1}}$ in $SL_n$. These sets define a stratification of $S$, and serve as analogs of double Bruhat cells in symmetric spaces. For example, they are $T$-leaves of $\pi_\theta$ by \cite{Lu14}*{Theorem~1.1}. 

For the diagonal case $(G\x G,\Delta(G))$, the usual double Bruhat cell $G^{u,v}$ of $G$ coincides with the symmetric space double Bruhat cell $\mathring{S}_{(u,v^{-1})}$. On the cluster side, the quiver for the corresponding double Bruhat cell in $G\x G$ is a disjoint union of two copies of the quiver for $G^{u,v}$, and the folding construction is simply identifying the two copies. By \cite{Bao26}*{Example~3.6}, the symmetric space has a total positivity structure that is compatible with that on $G$, which is well known to coincide with the cluster positivity structure.

\subsection{Symmetric matrices}
\label{sec:intro:symmatrix}

Let $M_n$ be the variety of $n\x n$ matrices over $\BC$ and $\Sym_n$ be the subvariety of symmetric matrices. A cluster structure on $M_n$ was constructed in \cite{FWZ21b}, where the cluster variables are generalized minors parametrized by double wiring diagrams. This cluster structure is a special case of that on more general types of reductive monoids \cite{SY25}.

Meanwhile, a Laurent phenomenon involving certain minors in $\Sym_n$ was established in \cite{STW16}, similar to the famous Laurent phenomenon of cluster algebras. This suggests that $\Sym_n$ should also have a cluster structure, which is previously unknown.

\subsection{Main results}
\label{sec:intro:results}

The main goal of this paper is to show that $SL_n/SO_n$ has a cluster structure. This gives the first results for cluster structures on symmetric spaces apart from the diagonal case, and we expect many of the results in this paper to hold for more general symmetric pairs.

\begin{thmintro}[Theorem~\ref{thm:MainThm} \& Corollary~\ref{cor:SLn/SOn}, \ref{cor:Symn}]
\label{thm:Main1}
The coordinate rings of the following varieties are cluster algebras:
\begin{itemize}
    \item the double Bruhat cells $\mathring{S}_w$ for the symmetric pair $(SL_n,SO_n)$;
    \item the symmetric space $SL_n/SO_n$;
    \item the variety of symmetric matrices $\Sym_n$.
\end{itemize}
Furthermore, the cluster algebras above coincide with their upper cluster algebras.
\end{thmintro}

We use the strata $\mathring{S}_w$ to study the cluster structures on symmetric spaces. This is a non-trivial choice as it allows us to take advantage of the Berenstein--Fomin--Zelevinsky cluster structures on $G^{u,v}$. We then use folding to obtain the seeds for $\mathring{S}_w$, based on calculations on the cluster minors in the seeds. This is a new application of folding, where we do not impose strong assumptions on the seeds.

The main difficulty of the proof is that the variety $\mathring{S}_w$ is not always factorial, so the same method as in \cite{BFZ05} does not apply. Specifically, we do not have the codimension $\geq2$ property immediately from the irreducibility of the cluster variables. We develop other methods to establish the cluster structure: we perform induction on $w$ similar to the deletion-contraction recurrence deployed in \cite{GLS26} and reduce to the minimal length elements in the conjugacy classes of the Weyl group.

Along the way, we show that the seeds for $\mathring{S}_w$ are locally acyclic, and the cluster structures are compatible with the Poisson structure on $S$. We also obtain cluster structures on $SL_n/SO_n$ and $\Sym_n$ by partial compactification, which are folded from that on $SL_n$ and $M_n$ respectively.

We note that the candidate seeds for $\mathring{S}_w$ constructed in Section~\ref{sec:cluster:seed} can be generalized directly to other quasi-split types except type $\text{AIII}_n$ with $n$ even. However, the exact method in this paper does not apply due to the different structures of the (twisted) conjugacy classes.

Finally, we record a cluster isomorphism between $\Sym_n$ and a Schubert cell for $Sp_{2n}$. Let $\CB$ be the flag variety for $Sp_{2n}$ and $\mathring{\CB}_{w^J}$ be the Schubert cell for $w^J$, the minimal length representative for the longest element in the quotient $C_n/A_{n-1}$. Cluster structures on Schubert cells were studied in \cite{GLS11}. We show that there is an isomorphism $\Phi:\Sym_n\isoto\mathring{\CB}_{w^J}$ respecting the cluster structures.

\vspace{.2cm}

\noindent {\bf Acknowledgment: } The author would like to thank Huanchen Bao for raising the conjecture in Remark~\ref{rem:Conjecture} and many valuable discussions leading to this project.

\section{Preliminaries}
\label{sec:prelim}

\subsection{Cluster algebras}
\label{sec:prelim:cluster}

A \textit{(labeled) seed} $\bfs$ is a quadruple $(J, J_\uf,(\ve_{ij})_{i,j\in J}, (d_i)_{i\in J})$ consisting of the following data:
\begin{itemize}
    \item An index set $J$, together with a subset $J_\uf\subset J$. The elements in $J$ are called vertices. The vertices in $J_\uf$ are called \textit{mutable} or \textit{unfrozen}. The vertices in $J_\fro=J\backslash J_\uf$ are called \textit{frozen}.
    \item A skew-symmetrizable matrix $(\ve_{ij})_{i,j\in J}$ over $\BQ$ with $\ve_{ij}\in \BZ$ if at least one of $i,j$ is in $J_\uf$, called the \textit{exchange matrix}.
    \item Positive rational numbers $d_i$ for $i\in J$ such that $d_i\ve_{ij}=-d_j\ve_{ji}$, called the \textit{symmetrizers}.
\end{itemize}

The exchange matrix is also interpreted as the adjacency matrix of a \textit{quiver}, which is a directed graph with no loops or oriented 2-cycles. The quiver will have weighted edges when the exchange matrix is not skew-symmetric.

Let $[a]_+=\max(a,0)$ for $a \in \BQ$. For any $k\in J_\uf$, the \textit{mutation at $k$} gives a new seed $\bfs'=\mu_k(\bfs)$ with the same $J$, $J_\uf$ and $(d_i)_{i \in J}$, and a new exchange matrix $(\ve'_{ij})_{i,j\in J}$ given by
\[
\ve'_{ij} = \begin{cases}-\ve_{ij}, & \text{if } k \in \{i,j\};\\
\ve_{ij}+[\ve_{ik}]_+\ve_{kj}+\ve_{ik}[-\ve_{kj}]_+, &\text{otherwise}.
\end{cases}
\]

We can repeat this process mutating at arbitrary $k\in J_\uf$. A seed $\bfs'$ is said to be \textit{mutation equivalent} to $\bfs$, denoted by $\bfs'\sim \bfs$, if it can be obtained from $\bfs$ by a sequence of mutations.

For any seed $\bfs_0$, the (abstract) upper cluster algebra with initial seed $\bfs_0$ is defined as follows. We fix the ambient field $K$ to be the field of rational functions over $\BC$ in $n$ variables, generated by the formal variables $A_i$ for $i\in J$. The variables $\{A_{i}=A_{i,\bfs_0}\}$ are called the \textit{cluster variables} in the seed $\bfs_0$. The variables $\{A_i \,|\, i\in J_\fro\}$ are  called \textit{frozen variables}.

Given the cluster variables $\{A_{i,\bfs}\}_{i \in J}$ in a seed $\bfs\sim \bfs_0$, the cluster variables $\{A_{i,\bfs'}\}_{i \in J}$ in $\bfs'=\mu_k(\bfs)$ are elements of $K$ defined by 
\[
\begin{cases}
A_{i,\bfs'}=A_{i,\bfs}, &\text{for $i\neq k$};\\
A_{k,\bfs'} = \frac{1}{A_{k,\bfs}}\left(\prod_{i\in J}A_{i,\bfs}^{[\ve_{ki}]_+}+\prod_{i\in J}A_{i,\bfs}^{[-\ve_{ki}]_+}\right), &\text{for $i = k$.}
\end{cases}
\]

The \textit{partially compactified cluster algebra} $\overline{\CA}(\bfs_0)$ is defined to be the $\BC$-algebra generated by the cluster variables over all seeds $\bfs\sim\bfs_0$. The \textit{cluster algebra} $\CA(\bfs_0)$ is obtained by inverting all the frozen variables in $\overline{\CA}(\bfs_0)$.

For any seed $\bfs\sim\bfs_0$, let $\LP(\bfs)=\BC[A_{i,\bfs}^{\pm1}]_{i\in J}$ and $\overline{\LP}(\bfs)=\BC[A_{i,\bfs}^{\pm1}, A_j]_{i\in J_\uf,j\in J_\fro}$.
The \textit{upper cluster algebra} $\CU(\bfs_0)$ is defined as the intersection
\[
\CU(\bfs_0)=\bigcap_{\bfs\sim \bfs_0}\LP(\bfs).
\]
The \textit{partially compactified upper cluster algebra} $\overline{\CU}(\bfs_0)$ is defined as the intersection
\[
\overline{\CU}(\bfs_0)=\bigcap_{\bfs\sim \bfs_0}\overline{\LP}(\bfs).
\]

It is clear that mutation equivalent seeds give canonically isomorphic upper cluster algebras, so the construction only depends on the mutation equivalence class of $\bfs_0$. By the Laurent phenomenon \cite{FZ03}*{Proposition~11.2}, all cluster variables belong to $\overline{\CU}$.

In this paper, we show that certain coordinate rings of varieties are isomorphic to cluster algebras. We will abuse notation and sometimes view the cluster variables as explicit functions on the variety.

For any seed $\bfs\sim \bfs_0$ and $i\in J$, let $\nu^\bfs_i$ be the valuation on $F$ induced from the valuation on $\LP(\bfs)$ given by the vanishing order of $A_{i,\bfs}$. By \cite{Qin25}*{Lemma~2.12}, for $i\in J_\fro$, $\nu^\bfs_i$ is independent of the seed $\bfs$, so we will omit the superscript. In particular, we have
\begin{equation}\label{eq:PartialCompactification}
\overline{\CU}=\{f\in\CU\mid \nu_i(f)\geq 0,\text{for }i\in J_\fro\}=\CU\cap \overline{\LP}(\bfs).
\end{equation}

We also recall the full rank and acyclic properties for upper cluster algebras. A seed $\bfs$ is said to have \textit{full rank} if the submatrix $(\ve_{ij})_{i\in J_\uf,j\in J}$ has full rank. In this case, the upper cluster algebra can be determined only using $\bfs$ and its adjacent seeds, i.e. the seeds obtained from $\bfs$ by applying a single mutation.

\begin{thm}[\cite{BFZ05}*{Corollary~1.9}]
\label{thm:FullRankUpperBound}
Let $\bfs$ be a seed with full rank. We have
\[
\CU(\bfs)=\LP(\bfs)\cap \bigcap_{k\in J_\uf}\LP(\mu_k(\bfs)).
\]
\end{thm}

A seed $\bfs$ is said to be \textit{acyclic} if its quiver does not contain oriented cycles of mutable vertices.

\begin{thm}[\cite{BFZ05}*{Theorem~1.18}]
\label{thm:AcyclicSeed}
Let $\bfs$ be an acyclic seed with full rank. Then $\CA(\bfs)=\CU(\bfs)$ and $\CA(\bfs)$ is generated by the cluster variables in $\bfs$ and its adjacent seeds along with the inverses of the frozen variables.
\end{thm}

A seed $\bfs$ is said to be \textit{locally acyclic} if $\Spec \CA(\bfs)$ has a finite cover by acyclic cluster charts, and in that case we have $\CA(\bfs)=\CU(\bfs)$ \cite{Mul13}.

\subsection{Double Bruhat cells}
\label{sec:prelim:dbc}

Let $G$ be a simply connected, connected simple algebraic group of rank $r$ over $\BC$. Let $B^+,B^-$ be two opposite Borel subgroups of $G$, $U^+,U^-$ their unipotent radical, $T=B^+\cap B^-$ the maximal torus, and $W$ the Weyl group.

Let $\{\alpha_i\}_{i\in I}$ be the simple roots, $\{\alpha_i^\vee\}_{i\in I}$ the simple coroots, $\{\omega_i\}_{i\in I}$ the fundamental weights. We have the Cartan matrix $(a_{ij})_{i,j\in I}$ given by $a_{ij}=\alpha_j(\alpha_i^\vee)$. We also have $\omega_j(\alpha_i^\vee)=\delta_{ij}$.

We fix a pinning $(T,B^+,B^-,x_i,y_i;i\in I)$ of $G$. For each $i\in I$, the maps
\begin{equation*}
\begin{pmatrix}
1&a\\
0&1
\end{pmatrix}\mapsto x_i(a),
\begin{pmatrix}
a&0\\
0&a^{-1}
\end{pmatrix}\mapsto \alpha_i^\vee(a),
\begin{pmatrix}
1&0\\
a&1
\end{pmatrix}\mapsto y_i(a)
\end{equation*}
define an embedding $SL_2\to G$. Let $\dot{s}_i=x_i(1)y_i(-1)x_i(1)$. The elements $\dot{s}_i$ satisfy the braid relations in $W$, so we can define the lift $\dot{w}\in G$ for any $w\in W$ using any reduced expression.

Let $\fg$ be the Lie algebra of $G$ with Killing form $(-,-)$, and $\fh$ be the Lie algebra of $T$. Let $\Phi$ be the root system of $\fg$, with $\Phi^+$ the subset of positive roots. We fix a Chevalley basis $E_\alpha$ for $\alpha\in \Phi$ and $H_i$ for $i\in I$.

Let $V(\lambda)$ be the highest weight representation of $G$ with highest weight $\lambda$, and $\eta_\lambda$ a choice of highest weight vector. For any $w\in W$, the extremal weight space of $w\lambda$ is 1-dimensional spanned by $\dot{w}\eta_\lambda$. Then for any $v,w\in W$, we define the \textit{generalized minor} $\Delta_{v\lambda,w\lambda}$ to be the regular function on $G$ given by
\[
\Delta_{v\lambda,w\lambda}(g)=\{ \dot{v}\eta_\lambda, g\dot{w}\eta_\lambda\},
\]
the coefficient of $\dot{v}\eta_\lambda$ in the weight space decomposition of $g\dot{w}\eta_\lambda$. Note that $\Delta_{v\lambda,w\lambda}=\Delta_{\lambda,\lambda}(\dot{v}^{-1}g\dot{w})$.

Alternatively, let $G_0=U^-TU^+$ be the open set in $G$ containing elements $g\in G$ that have Gaussian decomposition. The decomposition is unique and we will write it as $g=g_-g_0g_+$. For any dominant weight $\lambda$, and $g\in G_0$,  we have $\Delta_{\lambda,\lambda}(g)=\lambda(g_0)$, where $\lambda$ is viewed as a multiplicative character on $T$ such that $\lambda(\alpha_i^\vee(t))=t^{\lambda(\alpha_i^\vee)}$.

For $u,v\in W$, the \emph{double Bruhat cell} $G^{u,v}$ is defined by 
\[
G^{u,v}=B^+uB^+\cap B^-vB^-.
\]
It is known that $G^{u,v}$ is smooth of dimension $l(u)+l(v)+r$. We recall the cluster structure on the coordinate algebra $\BC[G^{u,v}]$ by Berenstein--Fomin--Zelevinsky \cite{BFZ05}.

Let $m=\ell(u)+\ell(v)$. A \emph{double reduced word} $\bfi$ for $(u,v)$ is a reduced word $(i_1,\dots,i_m)$ for $(u,v)$ in $W\x W$, in the alphabet $[-r,-1]\sqcup[1,r]$, where the simple reflections of the first copy of $W$ are denoted by $[-r,-1]$ and those for the second copy by $[1,r]$.

For any double reduced word $\bfi$ for $(u,v)$, we define a seed $\bfs(\bfi)$ as follows. The index set is $J=[-r,-1]\sqcup[1,m]$. Let $i_k=-k$ for $k\in [-r,1]$. For $k\in J$, let
\[
k^+=\min\{l\in J\mid l>k, |i_l|=|i_k|\},
\]
and let $k^+=m+1$ if there is no such $m$. An index $k\in J$ is frozen if either $k<0$ or $k^+>m$. The exchange matrix is defined as follows. For $k,l\in J$, let $p=\max(k,l)$, $q=\min(k^+,l^+)$. Then
\[
    \begin{split}
        \ve_{kl}= \begin{cases}
            -\sgn((k-l)i_p) &\text{if $p=q$;}\\
            -\sgn((k-l)i_p)a_{|i_k|,|i_l|} &\text{if $p<q$ and $i_pi_q(k-l)(k^+-l^+)>0$;}\\
            0 &\text{otherwise.}
        \end{cases}
    \end{split}
\]
The symmetrizers are given by $d_k=\frac{1}{2}( \alpha_{|i_k|},\alpha_{|i_k|})$.

For the fixed reduced word $\bfi$ for $(u,v)$, and $k\in [1,l]$, we denote
\[
u_{\leq k}=\prod_{\substack{t=1,\dots,k\\i_t<0}}s_{|i_t|},\qquad v_{>k}=\prod_{\substack{t=l,\dots,k+1\\i_t>0}}s_{|i_t|}.
\]
For $k\in[-r,-1]$, we let $u_{\leq k}=e$ and $v_{>k}=v^{-1}$. Now for $k\in J$, let
\begin{equation}\label{eq:ClusterVariableDefn}
\Delta(k;\bfi)=\Delta_{u_{\leq k}\omega_{|i_k|},v_{>k}\omega_{|i_k|}}.
\end{equation}

The generalized minors are viewed as functions on $G^{u,v}$ by restriction.

\begin{theorem}[\cite{BFZ05}*{Theorem~2.10}]
\label{thm:DBCcluster}
For any $(u,v)\in W\x W$ and double reduced word $\bfi$ for $(u,v)$, there is an isomorphism
\[
a_\bfi:\mathcal{U}(\bfs(\bfi))\longrightarrow \BC[G^{u,v}]
\]
given by $A_{k,\bfs(\bfi)}\mapsto \Delta(k;\bfi)$.
\end{theorem}

By \cite{SW21}*{Proposition~3.25}, different choices of $\bfi$ give mutation equivalent seeds $\bfs(\bfi)$, so the cluster structure on $G^{u,v}$ is independent of the choice of $\bfi$. We also have $\CA=\CU$ by \cite{SW21}*{Theorem~4.13}. By \cite{QY25}*{Appendix B}, when $u=v=w_0$, the partial compactification $\overline{\CU}(\bfs(\bfi))$ is isomorphic to $\BC[G]$.

For convenience later on, we also introduce an alternative index set $J'$ given by
\[
J'=\{(i,l)\,|\, i\in I, 0\leq l\leq n_i\},
\]
where $n_i$ is the number of appearances of $i$ and $-i$ in $\bfi$. Then there is a bijection between $J'$ and $J$, where $(i,0)$ corresponds to the index $-i$ in $J$, while for $l>0$, $(i,l)$ corresponds to the index $k>0$ such that $i_k$ is the $k$-th appearance of $i$ and $-i$ in $\bfi$. We order $J'$ using the order on $J$.

\begin{example}
\label{eg:running}
Our running example will be the open double Bruhat cell $G^{w_0,w_0}$ in $G=SL_3$. For $\bfi=(1,2,1,-1,-2,-1)$, the quiver for $\bfs(\bfi)$ is given by
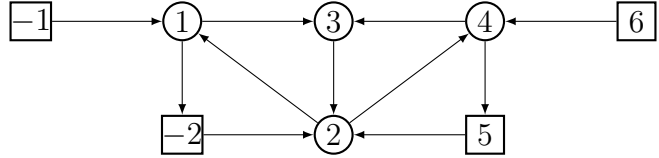
\begin{center}
\begin{tikzpicture}[every node/.style={inner sep=0, minimum size=0.5cm, thick, fill=white, draw}, x=2cm, y=1cm]
\node (0) at (0,0) {$-1$};
\node (1) at (1,0) [circle]{$1$};
\node (2) at (2,0) [circle]{$3$};
\node (3) at (3,0) [circle]{$4$};
\node (4) at (4,0) {$6$};
\node (5) at (1,-1.5) {$-2$};
\node (6) at (2,-1.5) [circle]{$2$};
\node (7) at (3,-1.5) {$5$};
\drawpath{0,1,2}{black}
\drawpath{4,3,2,6,1,5,6,3,7,6}{black}
\end{tikzpicture}
\end{center}
Let $g_{ij}$ be the matrix entries of $g\in SL_3$. Then the frozen variables in $\bfs(\bfi)$ are $A_{-2}=g_{12}g_{23}-g_{22}g_{13}$, $A_{-1}=g_{13}$, $A_5=g_{21}g_{32}-g_{22}g_{31}$, $A_6=g_{31}$, while the mutable cluster variables are $A_1=g_{12}$, $A_2=g_{11}g_{22}-g_{12}g_{21}$, $A_3=g_{11}$, $A_4=g_{21}$. 
\end{example}

\subsection{Folding of cluster algebras}
\label{sec:prelim:folding}

We follow \cite{FWZ21a}.

Let $\bfs$ be a seed and $\Gamma$ be a finite group acting on the indices $J$. Then $\bfs$ is said to be \textit{$\Gamma$-admissible} if it satisfies the following conditions:
\begin{enumerate}
    \item for any $i\in J$ and $\sigma\in \Gamma$, $i$ is mutable if and only if $\sigma(i)$ is, and $d_i=d_{\sigma(i)}$;
    \item for any $i,j\in J$ and $\sigma\in \Gamma$, $\ve_{ij}=\ve_{\sigma(i),\sigma(j)}$;
    \item for any $i\in J_\uf$ and $\sigma\in \Gamma$, $\ve_{i,\sigma(i)}=0$;
    \item for any $i\in J$, $j\in J_\uf$ and $\sigma\in \Gamma$, $\ve_{ij}\ve_{\sigma(i),j}\geq0$.
\end{enumerate}
In this case, we define the folded seed $\bfs^\Gamma$ as follows:
\begin{itemize}
    \item The index set $\overline{J}$ is the set of $\Gamma$-orbits in $J$. Using condition (1), we let $\overline{J}_\uf$ to be the $\Gamma$-orbits consisting of mutable vertices. We will use $\overline{i}$ to denote the orbit of $i\in J$.
    \item The exchange matrix $\ve_{\overline{i}\overline{j}}^\Gamma=\sum_{i'\in \overline{i}} \ve_{ij}$. By condition (2), $\ve_{\overline{i}\overline{j}}^\Gamma$ is independent of the choice of $j\in \overline{j}$.
    \item The symmetrizers $d_{\overline{i}}^\Gamma=d_i/|\overline{i}|$. We have
    \[
    d_{\overline{i}}^\Gamma\ve_{\overline{i}\overline{j}}^\Gamma =\frac{1}{|\Gamma|}\sum_{\sigma\in \Gamma}d_i\ve_{\sigma(i),j}
    =-\frac{1}{|\Gamma|}\sum_{\sigma\in \Gamma}d_j\ve_{j,\sigma(i)}
    =-d_{\overline{j}}^\Gamma\ve_{\overline{j}\overline{i}}^\Gamma.
    \]
\end{itemize}

By condition (3), for any $k,k'$ in the same mutable $\Gamma$-orbit, the mutations $\mu_k$, $\mu_{k'}$ commute. We define the orbit mutation $\mu_{\overline{k}}$ on $\bfs$ to be the composition of mutations in all directions in $\overline{k}$. We also have the mutation $\mu_{\overline{k}}$ on $\bfs^\Gamma$.

Mutation is compatible with folding under extra admissible assumptions.
\begin{lem}[\cite{FWZ21a}*{Lemma~4.4.5}]
\label{lem:FoldQuiverMutation}
Let $\bfs$ be $\Gamma$-admissible and $\overline{k}$ be a mutable $\Gamma$-orbit such that $\mu_{\overline{k}}(\bfs)$ is $\Gamma$-admissible. Then $\mu_{\overline{k}}(\bfs)^\Gamma=\mu_{\overline{k}}(\bfs^\Gamma)$.
\end{lem}

We now turn to the cluster variables. We fix an ambient field $F_\Gamma$ the field of rational functions over $\BC$ in $|\overline{J}|$ variables, generated by the formal variables $A_{\overline{i}}$ for $\overline{i}\in \overline{J}$. We have a quotient map $\pi:\LP(\bfs)\to \LP(\bfs^\Gamma)$ induced by $A_i\mapsto A_{\overline{i}}$.

\begin{lem}[\cite{FWZ21a}*{Lemma~4.4.8}]
\label{lem:FoldVariableMutation}
Let $\bfs$ be $\Gamma$-admissible and $\overline{k}$ be a mutable $\Gamma$-orbit. Then $\pi(A_{k,\mu_{\overline{k}}(\bfs)})=A_{\overline{k},\mu_{\overline{k}}(\bfs^\Gamma)}$.
\end{lem}

\begin{remark}
Note that the matching of the mutated cluster variables still follows from the proof without the extra assumption that $\mu_{\overline{k}}(\bfs)$ is $\Gamma$-admissible. In particular, we do not use the globally foldable assumption which is key in \citelist{\cite{Dup08}\cite{FWZ21a}}.
\end{remark}

\begin{cor}
\label{cor:FoldAcyclic}
Let $\bfs$ be $\Gamma$-admissible and acyclic. Then $\pi$ induces a surjection $\CA(\bfs)\to \CA(\bfs^\Gamma)$.
\end{cor}

\begin{proof}
If the quiver for $\bfs^\Gamma$ contains an oriented cycle of mutable vertices, we can lift it to a cycle in $\bfs$. This shows that if $\bfs$ is acyclic, so is $\bfs^\Gamma$. Then the surjection on cluster algebras follows immediately from Theorem~\ref{thm:AcyclicSeed} and Lemma~\ref{lem:FoldQuiverMutation}.
\end{proof}

\subsection{Symmetric spaces}
\label{sec:prelim:symmetric}

Let $\theta$ be an involution of $G$ and $K$ be the fixed-point subgroup.

Let $\eta:G/K \hookrightarrow G$ be the embedding given by $\eta(gK)=g\theta(g)^{-1}$. We denote the image by $S$, which is also identified with the twisted orbit of the identity, for the twisted action of $G$ on itself given by $g\cdot x= gx\theta(g)^{-1}$. $S$ is contained in the set $R = \{g\in G\,|\, g=\theta(g)^{-1}\}$. $R$ and $S$ are closed in $G$ by \cite{Spr85}.

In this paper we will assume $G=SL_n$ with $\theta$ given by the inverse transpose $\theta(g)=g^{-T}$. It is well known that any symmetric matrix in $SL_n$ can be written as $gg^T$ for some $g\in SL_n$, so $R=S$. We further fix a pinning such that $\theta(x_i(t))=y_i(-t)$ for any $i\in I$ and $t\in \BC$. We have $\theta(\dot{s}_i)=\dot{s}_i$ and 
$\theta(\alpha_i^\vee(t))=\alpha_i^\vee(t^{-1})$ for any $i\in I$.

Let $\mathring{S}_w = S\cap B^+wB^+$, which is affine and non-empty for any $w\in W$. By \cite{Lu14}, $\mathring{S}_w$ is irreducible and smooth with dimension
\[
\dim(G/K)+\ell(w)-\ell(w_0)=\ell(w)+\rank(G),
\]
and it is a $T$-leaf of a certain holomorphic Poisson structure on $G$.

Since $g=g^T$ for any $g\in S$, applying the transpose to $\mathring{S}_w$ gives
\[
\mathring{S}_w=S\cap B^-w^{-1}B^-.
\]
This shows that each $\mathring{S}_w$ is also the intersection of $S$ with the double Bruhat cell $G^{w,w^{-1}}$.

The variety $\mathring{S}_w$ is in general not factorial. For example, for $G=SL_2$, the open strata $\mathring{S}_{w_0}$ has coordinate ring isomorphic to $\BC[x,y,z^{\pm1}]/(xy=1+z^2)$. This is not a UFD, and the functions $x,y$ are irreducible but not prime.

We also have the following lemma on the generalized minors.
\begin{lem}[\cite{FZ99}*{Proposition~2.7}]
\label{lem:ThetaMinor}
For any $u,v\in W$ and $i\in I$,
\[
\Delta_{u\omega_i,v\omega_i}(g)=\Delta_{v\omega_i,u\omega_i}(g^T).
\]
\end{lem}

\section{Cluster structures on symmetric spaces}
\label{sec:cluster}

\subsection{The cluster structure}
\label{sec:cluster:seed}

In this section we propose cluster structures on the strata $\mathring{S}_w$ defined in Section~\ref{sec:prelim:symmetric}, with proofs given in the rest of Section~\ref{sec:cluster}.

Let $u,v\in W$ be such that $v^{-1}u=w$ with $\ell(w)=\ell(u)+\ell(v)$ and $\bfi$ be a double reduced word for $(u,v)$. Then $\bfj=(-\bfi^{-1},\bfi)$ is a double reduced word for $(w,w^{-1})$, where $-\bfi^{-1}$ is the reduced word obtained by reversing $\bfi$ and flipping the signs.

By Theorem~\ref{thm:DBCcluster}, $\bfj$ defines a seed for $G^{w,w^{-1}}$. Let $m_i$ be the number of appearances of $i$ and $-i$ in $\bfi$. Then in $\bfj$, $i$ and $-i$ appear $2m_i$ times.

\begin{lem}
\label{lem:FoldDBCSeed}
For $i\in I$ and $0\leq l\leq m_i$, the seed $\bfs(\bfj)$ is admissible for the involution $\sigma:(i,m_i-l)\leftrightarrow (i,m_i+l)$, and we have
\[
\Delta((i,m_i-l);\bfj)(g)=\Delta((i,m_i+l);\bfj)(g^T).
\]
\end{lem}

\begin{proof}
Recall definition of the seed $\bfs(\bfj)$ in Section~\ref{sec:prelim:dbc}. The quiver for $\bfs(\bfj)$ is obtained by gluing the quivers for $\bfs(-\bfi^{-1})$ and $\bfs(\bfi)$ along frozen variables, which become mutable in $\bfs(\bfj)$ with arrows added between them. Explicitly, these are the vertices $(i,m_i)$ for $i\in I$, with
\[
\ve_{(i,m_i),(j,m_j)}=
\begin{cases}
    -a_{ij} & \text{if $(i,m_i+1)<(j,m_j+1)$ and $(i,m_i+1)<(j,m_j+1)$};\\
    a_{ij} & \text{if $(i,m_i+1)>(j,m_j+1)$ and $(i,m_i+1)>(j,m_j+1)$};\\
    0 & \text{otherwise}.
\end{cases}
\]

The given involution is induced from the obvious isomorphism between the two quivers for $\bfs(-\bfi^{-1})$ and $\bfs(\bfi)$, and conditions (1), (2) and (3) for admissibility follow immediately.

For condition (4), since the diagram involution is trivial, the vertices involved in the gluing have a trivial orbit. Therefore, there is no oriented 4-cycle in the quiver for $\bfs(\bfj)$ involving two non-trivial orbits, and we have admissibility.

The second equality then follows from Lemma~\ref{lem:ThetaMinor}.
\end{proof}

Using the lemma, when restricting to $S$ where $g=\theta(g)^{-1}$, we have
\[
\Delta((i,m_i-l);\bfj)=\Delta((i,m_i+l);\bfj).
\]
Therefore, the quotient map coming from the folding of $\bfs(\bfj)$ is compatible with the quotient map $\BC[G]\to \BC[S]$. The same holds for the quotient map $\BC[G^{w,w^{-1}}]\to \BC[\mathring{S}_w]$. We will denote the folded seed by $\wt{\bfs}(\bfi)$.

\begin{lem}
\label{lem:IndepOfWord}
The seeds $\wt{\bfs}(\bfi)$ for all $u,v\in W$ such that $v^{-1}u=w$ with $\ell(w)=\ell(u)+\ell(v)$, and all double reduced words $\bfi$ for $(u,v)$ are mutation equivalent.
\end{lem}

\begin{proof}
Any two such $\bfi$ are related by successive applications of the following moves:
\begin{itemize}
    \item Replacing the first entry $i$ with $-i$.
    \item The usual moves on double reduced words: $(i,-j)\sim(-j,i)$ and $(i,j,i)\sim(j,i,j)$.
\end{itemize}

The first case corresponds to a single mutation at $(i,m_i)$ in $\bfs(\bfj)$ and the second case corresponds to mutations in each copy $\bfs(-\bfi^{-1})$ and $\bfs(\bfi)$ without involving $(i,m_i)$ (see \cite{FG06}*{Section 3} for the explicit mutation and computations).

It suffices to show that the mutations involved are compatible with folding, so the mutation equivalence descends from that of double Bruhat cells. Since each step involves only a single mutation, admissibility follows directly from Lemma~\ref{lem:FoldDBCSeed}, so Lemma~\ref{lem:FoldQuiverMutation} applies.
\end{proof}

The main goal of this section is to prove that we can fold the cluster structure on $G^{w,w^{-1}}$ to get a cluster structure on $\mathring{S}_w$.

\begin{thm}
\label{thm:MainThm} 
For any $u,v\in W$ such that $v^{-1}u=w$ with $\ell(w)=\ell(u)+\ell(v)$, and double reduced word $\bfi$ for $(u,v)$, there is an isomorphism $\CU(\wt{\bfs}(\bfi))\isoto \BC[\mathring{S}_w]$ such that the following diagram commutes
\begin{equation*}
    \begin{tikzcd}
    \CU(\bfs(\bfj)) \ar[d,"\pi"]\ar[r,"\sim"] & \BC[G^{w,w^{-1}}]\ar[d]\\
    \CU(\wt{\bfs}(\bfi)) \ar[r,"\sim"] & \BC[\mathring{S}_w]
    \end{tikzcd}
\end{equation*}
Furthermore, $\wt{\bfs}(\bfi)$ is locally acyclic. In particular, $\CA(\wt{\bfs}(\bfi))=\CU(\wt{\bfs}(\bfi))$.
\end{thm}

We will prove this theorem in the rest of Section~\ref{sec:cluster}. The proof involves induction on $w$ reducing to minimal length elements in the conjugacy classes of the Weyl group:
\begin{itemize}
    \item In Proposition~\ref{prop:InductionStep}, we show that if $\ell(s_iws_i)=\ell(w)+2$ and the theorem holds for $\mathring{S}_w$ and $\mathring{S}_{ws_i}$, then it holds for $\mathring{S}_{s_iws_i}$ as well.
    \item In Corollary~\ref{cor:ClusterReflection}, we show that if $s_iw>w$ and $ws_i>w$, then the theorem holds for $\mathring{S}_{s_iw}$ if and only if it holds for $\mathring{S}_{ws_i}$.
    \item The base cases are proved in Lemma~\ref{lem:BaseCase}, where we show that the theorem holds when each simple reflection appears at most once in $w$.
\end{itemize}

By the commutative diagram, the cluster variables in $\wt{\bfs}(\bfi)$ correspond to the restrictions of the cluster variables in $\bfs(\bfj)$. We will denote by $A_i^S$ the image of a cluster variable $A_i$ in $\BC[\mathring{S}_w]$ before the isomorphism in Theorem~\ref{thm:MainThm} is established.

\begin{example}
We continue with our running example. For $\bfi=(-1,-2,-1)$, the quiver $\bfs(\bfj)$ is the quiver given in Example~\ref{eg:running}. The involution on the quiver is the horizontal flip, and folding gives the following quiver for $\wt{\bfs}(\bfi)$
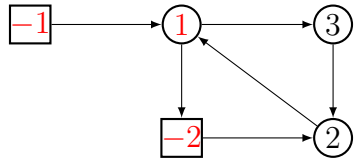
\begin{center}
\begin{tikzpicture}[every node/.style={inner sep=0, minimum size=0.5cm, thick, fill=white, draw}, x=2cm, y=1cm]
\node (0) at (0,0) {\red{$-1$}};
\node (1) at (1,0) [circle]{\red{$1$}};
\node (2) at (2,0) [circle]{$3$};
\node (5) at (1,-1.5) {\red{$-2$}};
\node (6) at (2,-1.5) [circle]{$2$};
\drawpath{0,1,2,6,1,5,6}{black}
\end{tikzpicture}
\end{center}
where the red vertices have symmetrizer $d_i=\frac{1}{2}$ and the black vertices have symmetrizer $d_i=1$.

The seed $\wt{\bfs}(\bfi)$ gives a cluster structure for $\mathring{S}_{w_0}$. The mutation relation $A_3A_3'=A_1A_4+A_2$ in $\bfs(\bfj)$ folds to $A_3A_3'=A_1^2+A_2$ in $\wt{\bfs}(\bfi)$, as we have $A_1=g_{12}=g_{21}=A_4$ on $S$. The cluster algebra for $\wt{\bfs}(\bfi)$ is of finite type $B_3$, so one can also list all the cluster variables and directly show that it is isomorphic to $\BC[\mathring{S}_{w_0}]$.
\end{example}

\begin{remark}
\label{rem:Conjecture}
The constructions and many results in Section~\ref{sec:cluster} have obvious generalizations to $G$ of general type and $\theta$ quasi-split. We conjecture that $\mathring{S}_w$ has an upper cluster structure for all symmetric pairs.
\end{remark}

\subsection{Folding full rank seeds}
\label{sec:cluster:fullrank}

By \cite{BFZ05}*{Proposition~2.6}, the seeds for double Bruhat cells have full rank. We offer another perspective to Theorem~\ref{thm:MainThm} using upper cluster algebras.

\begin{lem}
\label{lem:FullRankFolding}
Let $\bfs$ be $\Gamma$-admissible. If $\bfs$ has full rank, then so does $\bfs^\Gamma$.
\end{lem}

\begin{proof}
$(\ve_{ij})_{i\in J_\uf,j\in J}$ defines a linear map $L$ from $\BC^{|J|}$ to $\BC^{|J_\uf|}$. Full rank of $\bfs$ is equivalent to surjectivity of $L$.

When $\bfs$ is $\Gamma$-admissible, $L$ is equivariant under the $\Gamma$-action, and the corresponding map $L^\Gamma$ for $\bfs^\Gamma$ is the restriction of $L$ to the $\Gamma$-invariants in $\BC^{|J|}$, with basis given by sums of the standard basis vectors in $\BC^{|J|}$ in each orbit. A standard linear algebra exercise shows that if $L$ is surjective, so is $L^\Gamma$.
\end{proof}

Recall the quotient map $\pi:\LP(\bfs)\to \LP(\bfs^\Gamma)$. By Lemma~\ref{lem:FoldVariableMutation}, $\pi$ is compatible with the quotient maps $\LP(\mu_{\overline{k}}(\bfs))\to \LP(\mu_{\overline{k}}(\bfs^\Gamma))$ for any mutable $\Gamma$-orbit $\overline{k}$. When $\bfs^\Gamma$ has full rank,
\[
\CU(\bfs^\Gamma)=\LP(\bfs^\Gamma)\cap \bigcap_{\overline{k}\in \overline{J}\uf} \LP(\mu_{\overline{k}}(\bfs^\Gamma)).
\]
Any $f\in \CU(\bfs)$ lies in $\LP(\bfs)$ and $\LP(\mu_{\overline{k}}(\bfs))$ for any $\overline{k}$, so $\pi(f)\in \CU(\bfs^\Gamma)$. This shows that $\pi$ restricts to a map between the upper cluster algebras $\CU(\bfs)\to \CU(\bfs^\Gamma)$. Let $\CI$ be its kernel.

\begin{prop}
\label{prop:UpperCAKernel}
Under the isomorphism $\CA(\bfs(\bfj))=\CU(\bfs(\bfj))\simeq \BC[G^{w,w^{-1}}]$, $\CI$ is identified with the kernel of $\BC[G^{w,w^{-1}}]\to \BC[\mathring{S}_w]$.
\end{prop}

\begin{proof}
Let $\CI'$ be the kernel of $\BC[G^{w,w^{-1}}]\to \BC[\mathring{S}_w]$. $\CI$ is a prime ideal since the image of $\pi|_{\CU(\bfs)}$ is a subring of the integral domain $\CU(\bfs^\Gamma)$, while $\CI'$ is prime since $\mathring{S}_w$ is irreducible.

Now $\LP(\bfs(\bfj))$ is the localization of $\CU(\bfs(\bfj))$ over all mutable cluster variables $\bfs(\bfj)$. Then $\CI$ localize to the prime ideal in $\LP(\bfs(\bfj))$ generated by $A_i-A_{\sigma(i)}$ for all $i\in J$, where $\sigma$ is the involution in Lemma~\ref{lem:FoldDBCSeed}.

On the geometric side, let $\BT\subset G^{w,w^{-1}}$ be the torus defined by the seed $\bfs(\bfj)$. By Lemma~\ref{lem:FoldDBCSeed}, the transpose map permutes the coordinates of $\BT\simeq (\BC^\x)^{|J|}$, so the prime ideal for $\BT\cap \mathring{S}_w=\BT\cap R$ is also generated by $A_i-A_{\sigma(i)}$. This shows that $\CI$ is identified with $\CI'$.
\end{proof}

\subsection{Reduction map}
\label{sec:cluster:reduction}

Let $w\in W$ and $i\in I$ be such that $\ell(s_iws_i)=\ell(w)+2$. There is an isomorphism
\[
B^+s_iws_iB^+\xrightarrow{\sim}B^+s_iB^+\x^{B+} B^+ws_iB^+.
\]
The representative for the first factor can be chosen to be of the form $x_i(a)s_i$ for some unique $a\in \BC$. Therefore, for any $g\in \mathring{S}_{s_iws_i}$, there is a unique $t\in \BC$ such that $y_i(t)\dot{s}_i^{-1}g\in B^+ws_iB^+$. Then $g':=y_i(t)\dot{s}_i^{-1}g\dot{s}_ix_i(t)\in \mathring{S}_{ws_i}\sqcup \mathring{S}_w$.

Note that if $g'\in B^+ws_iB^+$, then $\Delta_{ws_i\omega_i,\omega_i}(g')\neq0$, while if $g\in B^+wB^+$, then we must have $\Delta_{ws_i\omega_i,\omega_i}(g')=0$. By construction, $g'\in \mathring{S}_{ws_i}\sqcup \mathring{S}_w$, so $g'\in \mathring{S}_{ws_i}$ if and only if $\Delta_{ws_i\omega_i,\omega_i}(g')\neq0$. In terms of $g$, we have
\[
\Delta_{ws_i\omega_i,\omega_i}(g')=\Delta_{ws_i\omega_i,\omega_i}(y_i(t)\dot{s}_i^{-1}g\dot{s}_ix_i(t))=\Delta_{s_iws_i\omega_i,s_i\omega_i}(g).
\]
Let $W,V\subset \mathring{S}_{s_iws_i}$ be non-vanishing and vanishing sets of $\Delta_{s_iws_i\omega_i,s_i\omega_i}$ respectively.

Let $\pi_W:W\to \BC^\x\x \mathring{S}_{ws_i}$ be given by
\[
\pi_W(g)=(\Delta_{s_iws_i\omega_i,\omega_i}(g),g'),
\]
and $\pi_V:V\to \BC\x \mathring{S}_w$
be given by
\[
\pi_V(g)=(\Delta_{s_iw\omega_i,\omega_i}(g),g').
\]
Let $\bfw$ be a reduced word for $w$. Consider the seeds $\bfs_0=\wt{\bfs}(-\bfw,-i,i)$ for $\mathring{S}_{s_iws_i}$, $\bfs_W=\wt{\bfs}(-\bfw,-i)$ for $\mathring{S}_{ws_i}$ and $\bfs_V=\wt{\bfs}(-\bfw)$ for $\mathring{S}_w$. Note that the function $\Delta_{s_iw\omega_i,s_i\omega_i}$ used to define $W$ and $V$ is a mutable cluster variable in $\bfs_0$, $\Delta_{s_iws_i\omega_i,\omega_i}$ is a frozen variable in $\bfs_0$, and $\Delta_{s_iw\omega_i,\omega_i}$ is the mutation of $\Delta_{s_iw\omega_i,s_i\omega_i}$ in $\bfs_0$.

The following proposition parallels \cite{GLS26}*{Theorem~3.13}. We modify the arguments to get around the UFD assumption.

\begin{prop}
\label{prop:WVDecomposition}
$\pi_W,\pi_V$ are isomorphisms which identify cluster variables in $\bfs_W$ and $\bfs_V$ to the restrictions of cluster variables in $\bfs_0$ to $W$ and $V$ respectively.
\end{prop}

\begin{equation*}
    \begin{tikzcd}
    (W,\bfs_0|_W) \ar[r,hook]\ar[d,"\pi_W"] & (\mathring{S}_{s_iws_i},\bfs_0) & (V,\bfs_0|_V) \ar[l,hook']\ar[d,"\pi_V"]\\
    \BC^\x \x (\mathring{S}_{ws_i},\bfs_W) &  & \BC\x (\mathring{S}_w,\bfs_V)
    \end{tikzcd}
\end{equation*}

\begin{proof}
For $\pi_W$, we have
\begin{align*}
\Delta_{s_iws_i\omega_i,\omega_i}(g)&=\Delta_{\omega_i,\omega_i}(\dot{s}_i^{-1}\dot{w}^{-1}y_i(-t)g'x_i(-t)\dot{s}_i^{-1})\\
&=-\Delta_{\omega_i,\omega_i}(\dot{s}_i^{-1}\dot{w}^{-1}g'y_i(-t))\\
&=t\Delta_{ws_i\omega_i,\omega_i}(g')-\Delta_{ws_i\omega_i,s_i\omega_i}(g')
\end{align*}
This shows that
\[
t=\frac{\Delta_{s_iws_i\omega_i,\omega_i}(g)+\Delta_{ws_i\omega_i,s_i\omega_i}(g')}{\Delta_{ws_i\omega_i,\omega_i}(g')}
\]
and $\pi_W$ is an isomorphism.

Similarly, for $\pi_V$, we have
\[
t=\frac{\Delta_{s_iw\omega_i,\omega_i}(g)+\Delta_{w\omega_i,s_i\omega_i}(g')}{\Delta_{w\omega_i,\omega_i}(g')},
\]
so $\pi_V$ is an isomorphism.

The matching of cluster variables follows from the equality
\begin{equation}
\label{eq:MatchVariable}
\Delta_{v\omega_j,\omega_j}(g')=\Delta_j(\dot{v}^{-1}y_i(t)s_i^{-1}gs_ix_i(t))=\Delta_{s_iv\omega_j,s_i\omega_j}(g)
\end{equation}
for any $v$ such that $s_iv>v$.
\end{proof}

Proposition~\ref{prop:WVDecomposition} is used in the induction step of Theorem~\ref{thm:MainThm}. We first need a few lemmas.

\begin{lem}
\label{lem:HartogsConverse}
Let $X=\Spec R$ be a Noetherian normal integral scheme and $U$ be open in $X$. If $\CO(U)=R$, then $X-U$ has codimension $\geq2$.
\end{lem}

\begin{proof}
By assumption, $R$ is equal to the intersection of all its localizations at height 1 primes. Suppose $X-U$ contains a height 1 prime $\fp$. It suffices to find $f\in K(R)$ which is not contained in $R_\fp$ but is contained in $R_\fq$ for all other height 1 primes of $R$, so that $R\subsetneq\CO(U)$.

Since $R$ is normal, $R_\fp$ is a DVR, let $x\in R$ be a uniformizer. Then $\nu_\fp(x^{-1})=-1$, $\nu_{\fq_i}(x^{-1})=0$ for all but finitely many primes $\fq_i$, for which $\nu_{\fq_i}(x^{-1})<0$. For each $\fq_i$, let $y_i\in \fq_i\backslash\fp$, so $\nu_\fp(y_i)=0$ and $\nu_{\fq_i}(y_i)>0$. Then for $d_i$ large enough, $f=x^{-1}\prod y_i^{d_i}$ satisfies $\nu_\fp(f)=-1$ and $\nu_\fq(f)\geq0$ for all $\fq\neq \fp$ as desired.
\end{proof}

\begin{cor}
\label{cor:ClusterCodim2}
Let $\bfs$ be a seed with full rank, and suppose $X=\Spec \CU(\bfs)$ is Noetherian. Let $Z_{ij}=\{A_i=A_j=0\}$ for distinct $i,j\in J_\uf$ and $Z_{kk}=\{A_k=\mu_k(A_k)=0\}$ for $k\in J_\uf$. Then $Z_{ij}$ and $Z_{kk}$ all have codimension $\geq 2$ in $X$.
\end{cor}

\begin{proof}
$\CU(\bfs)$ is integrally closed by \cite{Mul13}*{Proposition~2.1}. Let $U$ be the union of the tori in $X$ defined by $\bfs$ and its adjacent seeds. Then $X-U$ is the union of the $Z_{ij}$ and $Z_{kk}$. By Theorem~\ref{thm:FullRankUpperBound}, $\CU(\bfs)=\CO(U)$. Then Lemma~\ref{lem:HartogsConverse} shows that $X-U$ has codimension $\geq2$.
\end{proof}

The codimension $\geq2$ property is nontrivial for $\mathring{S}_w$ because $\mathring{S}_w$ is in general not factorial. After establishing it, we can then use the same method as in \cite{BFZ05} to prove the cluster structures. The argument is summarized by the following lemma.

For any $w\in W$ and seed $\bfs=\wt{\bfs}(\bfi)$ for $\mathring{S}_w$, let $Z_{ij}^S=\{A_i^S=A_j^S=0\}$ for distinct $i,j\in J_\uf$ and $Z_{kk}^S=\{A_k^S=\mu_k(A_k^S)=0\}$ for $k\in J_\uf$ be the corresponding sets in $\mathring{S}_w$ for the proposed cluster structure.

\begin{lem}
\label{lem:Codim2Cluster}
If $Z_{ij}^S$ and $Z_{kk}^S$ all have codimension $\geq2$ in $\mathring{S}_w$, then $\CU(\bfs)\simeq \BC[\mathring{S}_w]$.
\end{lem}

\begin{proof}
The seed $\bfs$ has full rank by Lemma~\ref{lem:FullRankFolding}. By Theorem~\ref{thm:FullRankUpperBound}, $\CU(\bfs)$ is equal to the intersection of the Laurent polynomial rings for $\bfs$ and its adjacent seeds $\mu_k(\bfs)$ for $k\in J_\uf$.

Now we study the corresponding functions on $\mathring{S}_w$. By restriction and folding of the mutation relations, the functions $A_k^S$ for $k\in J_\fro$ are non-vanishing on $\mathring{S}_w$, and for each $k\in J_\uf$, $\mu_k(A_k^S)$ is regular on $\mathring{S}_w$. Furthermore, since $R=S$, the seed $\bfs$ and its adjacent seeds define the tori $\BT_\bfs$ and $\BT_{\mu_k(\bfs)}$ in $\mathring{S}_w$, which are equal to the intersection of the corresponding tori in $G^{w,w^{-1}}$ with $R$.

Since $Z_{ij}^S$ and $Z_{kk}^S$ all have codimension $\geq2$, so does the complement
\[
\mathring{S}_w-\left(\BT_\bfs \cup \bigcup_{k\in J_\uf} \BT_{\mu_k(\bfs)}\right).
\]
Therefore, $\BC[\mathring{S}_w]$ is also equal to the intersection of the same Laurent polynomial rings, and the desired isomorphism follows.
\end{proof}

\begin{prop}
\label{prop:InductionStep}
Let $w\in W$ and $i\in I$ be such that $\ell(s_iws_i)=\ell(w)+2$. If Theorem~\ref{thm:MainThm} holds for $\mathring{S}_w$ and $\mathring{S}_{ws_i}$, then it holds for $\mathring{S}_{s_iws_i}$ as well.
\end{prop}

\begin{proof}
Let $A_k^S$ be the mutable cluster variable $\Delta_{s_iw\omega_i,s_i\omega_i}$ in $\bfs_0$. We first consider the sets $Z_{ij}^S$ and $Z_{jj}^S$ in $\mathring{S}_{s_iws_i}$.

For any $i\neq k\in J_\uf$, the set $Z_{ik}^S$ is contained in $V$. Under $\pi_V$, $Z_{ik}^S$ corresponds to the vanishing set of a cluster variable in $\BC\x \mathring{S}_w$. This shows that $Z_{ik}^S$ has codimension $\geq1$ in $V$, hence codimension $\geq2$ in $\mathring{S}_w$. The same goes for $Z_{kk}^S$, since $\mu_k(A_k^S)$ corresponds to the coordinate of the first component in $\BC\x \mathring{S}_w$.

Meanwhile, for $i,j\neq k$ (allowing $i=j$), we decompose $Z_{ij}^S$ into $Z_{ij}^S\cap W$ and $Z_{ij}^S\cap V$. $Z_{ij}^S\cap V$ is contained in $Z_{ik}^S$ and has codimension $\geq 2$. Under $\pi_W$, $Z_{ij}^S\cap W$ corresponds to a set $Z'_{ij}$ in $\BC^\x\x \mathring{S}_{ws_i}$. Since $\mathring{S}_{ws_i}$ is irreducible and smooth, Corollary~\ref{cor:ClusterCodim2} shows that $Z'_{ij}$ has codimension $\geq2$ in $W$, so $Z_{ij}^S$ has codimension $\geq2$ in $\mathring{S}_w$ as well. By Lemma~\ref{lem:Codim2Cluster}, we have $\CU(\bfs_0)\simeq \mathring{S}_{s_iws_i}$.

For local acyclicity, note that the vertex $k$ for $\Delta_{s_iw\omega_i,s_i\omega_i}$ is a sink in the quiver $\bfs_0$, i.e. it has no arrows pointing to mutable vertices. Freezing $k$ in $\bfs_0$ gives the seed $\bfs_W$ plus an isolated frozen vertex. If we freeze all mutable cluster variables in $\bfs_0$, the mutable part of the resulting seed is the same as that of $\bfs_V$ plus an isolated vertex. In particular, by the induction hypothesis, both freezings give locally acyclic seeds, which have $\CA=\CU$.

Let $f$ be the product of all mutable cluster variables adjacent to $k$. By \cite{Mul13}*{Lemma~3.4}, $A_k$ and $f$ define cluster localizations of $\CU(\bfs_0)$. Since $k$ is a sink, the mutation rule at $k$ shows that $A_k$ and $f$ cannot vanish at the same time. Therefore, $A_k$ and $f$ define a cluster chart on $\mathring{S}_{s_iws_i}$, and $\bfs_0$ is locally acyclic.
\end{proof}

\subsection{Reflection map}
\label{cluster:reflection}

The other ingredient we will use is the reflection map introduced in \cite{SW21}.

Let $w\in W$ and $i\in I$ be such that $s_iw>w$ and $ws_i>w$. The same construction $g':=y_i(t)\dot{s}_i^{-1}g\dot{s}_ix_i(t)$ as in Section~\ref{sec:cluster:reduction} is an isomorphism $ \mathring{S}_{s_iw}\isoto \mathring{S}_{ws_i}$. We will denote this map by $r_i$.

\begin{prop}
Let $\bfw$ be a reduced word for $w$ and consider the seeds $\bfs_1=\wt{\bfs}(-\bfw,i)$ for $\mathring{S}_{s_iw}$ and $\bfs_2=\wt{\bfs}(-\bfw,-i)$ for $\mathring{S}_{ws_i}$. The map $r_i$ matches the images of cluster variables in $\bfs_1$ and $\bfs_2$ in $\BC[\mathring{S}_{s_iw}]$ and $\BC[\mathring{S}_{ws_i}]$ respectively, except for the frozen variables $\Delta_{s_iw\omega_i,\omega_i}$ in $\bfs_1$ and $\Delta_{ws_i,\omega_i,\omega_i}$ in $\bfs_2$, where we instead have
\[
r_i^*(\Delta_{ws_i\omega_i,\omega_i})=\Delta_{s_iw\omega_i,\omega_i}^{-1}\prod_{j\neq i}\Delta_{s_iw\omega_j,\omega_j}^{-a_{ij}}.
\]
\end{prop}

\begin{proof}
We have
\[
\Delta_{ws_i\omega_i,\omega_i}(r_i(g))=\Delta_{ws_i\omega_i,\omega_i}(y_i(t)\dot{s}_i^{-1}g\dot{s}_ix_i(t))=\Delta_{ws_i\omega_i,s_i\omega_i}(y_i(t)\dot{s}_i^{-1}g).
\]
By \cite{FZ99}*{Theorem~1.17}, we have the identity
\[
\Delta_{w\omega_i,\omega_i}\Delta_{ws_i\omega_i,s_i\omega_i}=\Delta_{ws_i\omega_i,\omega_i}\Delta_{w\omega_i,s_i\omega_i}+\prod_{j\neq i}\Delta_{w\omega_i,\omega_i}^{-a_{ij}}.
\]
Since $y_i(t)\dot{s}_i^{-1}g\in B^+wB^+$, we have
\begin{align*}
\Delta_{ws_i\omega_i,\omega_i}(y_i(t)\dot{s}_i^{-1}g)&=0,\\
\Delta_{w\omega_i,\omega_i}(y_i(t)\dot{s}_i^{-1}g)&\neq0.
\end{align*}
This gives
\begin{align*}
\Delta_{ws_i\omega_i,s_i\omega_i}(r_i(g))&=\Delta_{w\omega_i,\omega_i}(y_i(t)\dot{s}_i^{-1}g)^{-1}\prod_{j\neq i}\Delta_{w\omega_i,\omega_i}^{-a_{ij}}(y_i(t)\dot{s}_i^{-1}g)^{-a_{ij}}\\
&=\Delta_{s_iw\omega_i,\omega_i}(g)^{-1}\prod_{j\neq i}\Delta_{s_iw\omega_i,\omega_i}^{-a_{ij}}(g)^{-a_{ij}}.
\end{align*}

The matching of the other cluster variables is again (\ref{eq:MatchVariable}).
\end{proof}

Comparing with the exchange matrices for $\bfs_1$ and $\bfs_2$, on the cluster side, we have a quasi-cluster transformation $\CU(\bfs_1)\simeq \CU(\bfs_2)$ in the sense of \cite{Fra16}. This gives the following corollary.

\begin{cor}
\label{cor:ClusterReflection}
Let $w\in W$ and $i\in I$ be such that $s_iw>w$ and $ws_i>w$. Then Theorem~\ref{thm:MainThm} holds for $\mathring{S}_{s_iw}$ if and only if it holds for $\mathring{S}_{ws_i}$.
\end{cor}

\subsection{Proof of Theorem~\ref{thm:MainThm}}
\label{sec:cluster:proof}

\begin{lem}
\label{lem:BaseCase}
Let $w\in W$ be such that each simple reflection appears at most once in $w$. Then Theorem~\ref{thm:MainThm} holds for $\mathring{S}_w$.
\end{lem}

\begin{proof}
By construction, the mutable part of the quiver for $\bfs(\bfj)$ for $G^{w,w^{-1}}$ is obtained by adding an orientation to a subgraph of the Dynkin diagram of $G$, so $\bfs(\bfj)$ is acyclic and we have a surjection $\CA(\bfs(\bfj))\to \CA(\wt{\bfs}(\bfi))$ by Corollary~\ref{cor:FoldAcyclic}. The quiver for $\wt{\bfs}(\bfi)$ is also acyclic, so we have $\CA(\wt{\bfs}(\bfi))=\CU(\wt{\bfs}(\bfi))$ by \cite{BFZ05}*{Theorem~1.18}.

By Proposition~\ref{prop:UpperCAKernel}, the kernel of the surjection
\[
\CU(\bfs(\bfj))=\CA(\bfs(\bfj))\to\CA(\wt{\bfs}(\bfi))=\CU(\wt{\bfs}(\bfi))
\]
is identified with the kernel of $\BC[G^{w,w^{-1}}]\to \BC[\mathring{S}_w]$, so we conclude that $\CU(\wt{\bfs}(\bfi))\simeq \BC[\mathring{S}_w]$.
\end{proof}

Note that in this case, the folding is simply changing the frozen part of the seed, so we can also prove Lemma~\ref{lem:BaseCase} by base change.

\begin{proof}[Proof of Theorem~\ref{thm:MainThm}]
Consider the conjugacy classes $C_w=\{vwv^{-1}|v\in W\}$ in $W$. By \cite{GP93}*{Theorem~1.1}, any $w\in W$ can be reduced to an element $w'\in C_w$ with minimal length via cyclic shifts $x\mapsto x'=s_ixs_i$ where $\ell(x')\leq \ell(x)$. Combining Proposition~\ref{prop:InductionStep} and Corollary~\ref{cor:ClusterReflection}, we reduce to proving the theorem for all minimal elements in the conjugacy classes of $W$.

Each conjugacy class in $W\simeq \mathring{S}_n$ contains an element of the form
\[
w_{\min} = (1,\cdots,k_1)(k_1+1,\cdots ,k_2)\cdots(k_{l-1}+1,\cdots,k_l)
\]
for $1\leq k_1<k_2\cdots <k_l=n$ in disjoint cycle notation. We see that each simple reflection appears at most once in $w_{\min}$, which also implies $w_{\min}$ is minimal length. By Lemma~\ref{lem:BaseCase}, the result follows.
\end{proof}

\section{Cluster Poisson structures}
\label{sec:poisson}

In this section, we show that the cluster structures on $\mathring{S}_w$ constructed in Section~\ref{sec:cluster} are compatible with the Poisson structure in \cite{Lu14}.

\subsection{Poisson structures on $G$}
\label{sec:poisson:group}

We first recall the standard and twisted Poisson structures on $G$, following \citelist{\cite{EL07}\cite{Lu14}}.

For any $x\in \fg$, let $x^L$ and $x^R$ be respectively the left and right invariant vector fields on $G$ with value $x$ at identity of $G$.

The standard Poisson structure $\pi_{\st}$ endows $G$ with a Poisson--Lie group structure. It is given by the bivector field
\[
\pi_\st=\frac{1}{2}\sum_{\alpha\in \Phi^+}(E_\alpha^R\wedge E_{-\alpha}^R-E_\alpha^L\wedge E_{-\alpha}^L).
\]

Meanwhile, for any automorphism $\theta$ of $G$, there is a Poisson structure $\pi_\theta$ on $G$ such that each $\theta$-twisted conjugacy class is a Poisson submanifold. When $\theta(B^+)=B^+$ and $\theta(T)=T$, an explicit formula for $\pi_\theta$ is given in \cite{Lu14}*{Lemma~2.4}. Let $y_i$, $i\in I$ be a basis for $\fh$ such that $( y_i,y_j) =\frac{1}{2}\delta_{ij}$. The Poisson structure $\pi_\theta$ is given by the bivector field
\[
\pi_\theta=\sum_{i\in I}\theta(y_i)^L\wedge y_i^R+\sum_{\alpha\in \Phi^+} \theta(E_{-\alpha})^L\wedge E_\alpha^R+\frac{1}{2}\sum_{\alpha\in \Phi^+}(E_\alpha^R\wedge E_{-\alpha}^R+E_\alpha^L\wedge E_{-\alpha}^L).
\]

When $\theta$ is arbitrary, one reduces to the $\theta$-stable case by applying an extra conjugation. Let $g_0\in G$ be such that $\theta'=\Ad_{g_0}\circ\theta$ has $\theta'(B^+)=B^+$ and $\theta'(T)=T$. The Poisson structure $\pi_\theta$ is then defined by $(r_{g_0})_*\pi_{\theta'}$, where $r_{g_0}$ is right multiplication by $g_0$.

We apply these results to $G=SL_n$ and $\theta(g)=g^{-T}$. In this case, $\theta'(g)=\dot{w}_0g^{-T}\dot{w}_0^{-1}$ stabilizes $B^+$ and $T$. In fact $\theta'$ is the Dynkin involution of $G$ with respect to the pinning. Using $\pi_{\theta'}$, we compute that
\begin{equation}
\label{eq:TwistedPoisson}
\pi_\theta=\sum_{i\in I}y_i^L\wedge y_i^R-\sum_{\alpha\in \Phi^+} E_\alpha^L\wedge E_\alpha^R+\pi_\st.
\end{equation}

\subsection{Cluster Poisson structures}
\label{sec:poisson:cluster}

We then recall compatible Poisson structures on cluster algebras following \citelist{\cite{BZ05}\cite{GSV10}}.

Let $\bfs_0$ be a seed. A Poisson bracket $\{-,-\}$ on $\CA=\CA(\bfs_0)$ is said to be \textit{compatible} with the cluster structure if for any seed $\bfs\sim \bfs_0$, the cluster variables $\{A_{j,\bfs}\}_{j\in J}$ are log canonical, i.e.
\[
\{A_{i,\bfs},A_{j,\bfs}\}=\Omega_{ij,\bfs}A_{i,\bfs}A_{j,\bfs}
\]
for some constants $\Omega_{ij,\bfs}$. The matrix $(\Omega_{ij,\bfs})$ is skew-symmetric.

The compatibility condition can be checked in a single seed, where one needs to verify the cluster variables are log canonical, and for any $j\in J_\uf$, $i\in J$,
\begin{equation}
\label{eq:CompatibleEqn}
\sum_{k\in J}\ve_{kj}\Omega_{ki}=\delta_{ij}p_j
\end{equation}
for some non-zero $p_j$.

Using (\ref{eq:CompatibleEqn}), one can then prove that for any $k\in J_\uf$, the cluster variables in $\mu_k(\bfs)$ are compatible, and the same equation holds for $\Omega_{ij,\mu_k(\bfs)}$ and $\ve_{ij,\mu_k(\bfs)}$, thereby recovering compatibility for all seeds.

\subsection{Compatibility for $\mathring{S}_w$}
\label{sec:poisson:compatible}

The standard Poisson structure on $G$ is shown to be compatible with the cluster structure in \cite{GSV10}. The key computation is summarized in Corollary 4.21, which shows that for the cluster variables $A_k=\Delta(k;\bfi)$, if $j<k$, then
\begin{equation}
\label{eq:StdBracket}
\{A_j,A_k\}_\st=\frac{( u_{\leq j}\omega_{|i_j|},u_{\leq k}\omega_{|i_k|})-( v_{> j}\omega_{|i_j|},v_{> k}\omega_{|i_k|})}{2}A_jA_k.    
\end{equation}
It is then shown that the coefficients $\Omega_{jk}$ are compatible with the exchange matrix.

We now turn to $\mathring{S}_w$. The cluster variables in our seeds defined in Section~\ref{sec:cluster:seed} are all obtained by restricting generalized minors on $G$. Since $S$ is a Poisson submanifold of $(G,\pi_\theta)$, we can compute their Poisson bracket $\{-,-\}_S$ by restricting the Poisson bracket $\{-,-\}_\theta$ on $G$.

Let $\bfi$ be a reduced word for $w$ and consider the seed $\wt{\bfs}(-\bfi)$ for $\mathring{S}_w$. The cluster variables are of the form $\Delta_{u\omega_i,\omega_i}$ for $u\in W$, parametrized by $(i,l)$ with $l\geq m_i$.

\begin{lem}
\label{lem:PoissonBracketS}
Let $A_j,A_k$ be cluster variables in $\wt{\bfs}(-\bfi)$. Then
\[
\{A_j,A_k\}_S=(\Omega_{jk}+\Omega_{\sigma(j),k})A_jA_k.
\]
\end{lem}

\begin{proof}
Let $\langle-,-\rangle$ denote the pairing between $\fh^*$ and $\fh$. We compute directly
\begin{align*}
y_i^L(\Delta_{u\omega_j,\omega_j})(g)&=\langle\omega_j,y_i\rangle,\\
y_i^R(\Delta_{u\omega_j,\omega_j})(g)&=\langle u\omega_j,y_i\rangle,
\end{align*}
and for any $\alpha\in \Phi^+$,
\[
E_\alpha^L(\Delta_{u\omega_j,\omega_j})=0.
\]
Therefore, as functions on $G$,
\begin{align*}
\{A_j,A_k\}_\theta&=\{A_j,A_k\}_\st+\sum_{i\in I}(\langle\omega_{|i_j|},y_i\rangle\langle u_{\leq k}\omega_{|i_k|},y_i\rangle - \langle u_{\leq j}\omega_{|i_j|},y_i\rangle \langle \omega_{|i_k|},y_i\rangle )A_jA_k\\
&=\{A_j,A_k\}_\st+\frac{( \omega_{|i_j|}, u_{\leq k}\omega_{|i_k|})-( u_{\leq j}\omega_{|i_j|},\omega_{|i_k|})}{2}A_jA_k
\end{align*}
Note that the coefficient for the second term is the Poisson bracket $\{A_{\sigma(j)},A_k\}_\st$, so
\[
\{A_j,A_k\}_\theta=(\Omega_{jk}+\Omega_{\sigma(j),k})A_jA_k,
\]
and the result follows by restriction to $S$.
\end{proof}

Now we turn to folding of seeds with compatible Poisson structures. Let $\bfs$ be admissible and suppose the matrix $\Omega=(\Omega_{ij})_{i,j\in J}$ is compatible with $(\ve_{ij})$. Same as the exchange matrix, we impose the admissibility condition $\Omega_{ij}=\Omega_{\sigma(i),\sigma(j)}$ for any $i,j\in J$ and $\sigma\in \Gamma$. Note that this implies $p_j=p_{\sigma(j)}$ for any $j\in J$ and $\sigma\in \Gamma$.

Define $\Omega^\Gamma$ by
\[
\Omega^\Gamma_{\overline{i}\overline{j}}=\sum_{\sigma\in \Gamma}\Omega_{\sigma(i),j}=\sum_{\sigma\in \Gamma}\Omega_{i,\sigma(j)}.
\]
$\Omega^\Gamma$ is again independent of the choice of $j$, and it is skew-symmetric.

\begin{lem}
\label{lem:FoldCompatible}
$\Omega^\Gamma$ is compatible with $(\ve^\Gamma_{\overline{i}\overline{j}})$ in the folded seed $\bfs^\Gamma$.
\end{lem}

\begin{proof}
For any $\overline{j}\in \overline{J}_\uf$, and $\overline{i}\in \overline{J}$, we have
\[
\sum_{\overline{k}\in \overline{J}}\ve^\Gamma_{\overline{k}\overline{j}}\Omega^\Gamma_{\overline{k}\overline{i}}=\sum_{\overline{k}\in \overline{J}}\sum_{k'\in\overline{k}}\sum_{\sigma\in \Gamma}\ve_{k',j}\Omega_{k,\sigma(i)}
=\sum_{\sigma\in \Gamma}\delta_{\sigma(i),j}p_j
=\delta_{\overline{i}\overline{j}}p_j\frac{|\Gamma|}{|\overline{j}|},
\]
so we obtain (\ref{eq:CompatibleEqn}) for $\bfs^\Gamma$.
\end{proof}

\begin{prop}
\label{prop:SwCompatible}
The Poisson structure on $\mathring{S}_w$ is compatible with the cluster structure.
\end{prop}

\begin{proof}
Combining Lemma~\ref{lem:PoissonBracketS} and \ref{lem:FoldCompatible}, it suffices to check that $\Omega$ is admissible for the seed $\bfs(\bfi^{-1},-\bfi)$. Let $j<k$. If $\sigma(j)>\sigma(k)$, then the result follows immediately from (\ref{eq:StdBracket}) by
\[
\Omega_{\sigma(j),\sigma(k)}=-\Omega_{\sigma(k),\sigma(j)}=
-\frac{( v_{> k}\omega_{|i_k|},v_{> j}\omega_{|i_j|})-( u_{\leq k}\omega_{|i_k|},u_{\leq j}\omega_{|i_j|})}{2}=\Omega_{jk}.
\]
Otherwise, $A_j$ and $A_k$ must be of the same form, either both $\Delta_{w\omega_i,\omega_i}$, both $\Delta_{\omega_i,\omega_i}$, or both $\Delta_{\omega_i,w\omega_i}$, for various $i\in I$. In any case, we have $\Omega_{jk}=\Omega_{\sigma(j),\sigma(k)}=0$.
\end{proof}

\section{Partial compactifications}
\label{sec:partialcpt}

\subsection{Partial compactifications of foldings}
\label{sec:partialcpt:folding}

In this section we study how partial compactification interacts with foldings.

Let $X$ be an affine variety with $\BC[X]= \overline{\CU}(\bfs)$, $U\subset X$ be the affine open subvariety with $\BC[U]= \CU(\bfs)$. Suppose we have a folding $\bfs^\Gamma$ such that the map $\pi:\CU(\bfs)\to \CU(\bfs^\Gamma)$ is surjective. Then $\pi$ defines a subvariety $Y_U$ closed in $U$, with $\BC[Y_U]= \CU(\bfs^\Gamma)$. Let $Y$ be the closure of $Y_U$ in $X$, so $\BC[Y]$ is a subalgebra of $\CU(\bfs^\Gamma)$.

\begin{prop}
\label{prop:CptFold}
If the frozen variables of $\bfs^\Gamma$ are prime in $\BC[Y]$, then $\BC[Y]= \overline{\CU}(\bfs^\Gamma)$.
\end{prop}

\begin{proof}
The projection $\pi$ restricts to a projection $\overline{\LP}(\bfs)\to \overline{\LP}(\bfs^\Gamma)$, so we also get a map $\overline{\CU}(\bfs)\to\overline{\CU}(\bfs^\Gamma)$. Any function $f\in \BC[Y]$ is the restriction of a function on $\BC[X]$. By the assumption $\BC[X]\simeq \overline{\CU}(\bfs)$, we have $f\in \overline{\CU}(\bfs^\Gamma)$ so $\BC[Y]\subset \overline{\CU}(\bfs^\Gamma)$. Then the equality follows from the classical version of \cite{QY25}*{Theorem~C}.
\end{proof}

Note that by \cite{QY25}*{Lemma~5.3}, any frozen variable in a partially compactified upper cluster algebra $\overline{\CU}$ is prime, so the assumption for primeness in $\BC[Y]$ in Proposition~\ref{prop:CptFold} is necessary.

\subsection{Applications}
\label{sec:partialcpt:appl}

We now apply Proposition~\ref{prop:CptFold} to the cluster structures on $\mathring{S}_w$. Using Theorem~\ref{thm:MainThm}, the only assumption we need to check is the primeness of the frozen variables.

\subsubsection{The symmetric space $SL_n/SO_n$}
\label{sec:partialcpt:appl:SLnSOn}

The first example is the symmetric space $SL_n/SO_n$, identified as the subvariety $S\subset G$. Recall the cluster structures on $\BC[G]$, $\BC[G^{w_0,w_0}]$ and $\BC[\mathring{S}_{w_0}]$.

\begin{cor}
\label{cor:SLn/SOn}
For any $u,v\in W$ such that $v^{-1}u=w_0$ with $\ell(w_0)=\ell(u)+\ell(v)$, and double reduced word $\bfi$ for $(u,v)$, we have an isomorphism $\overline{\CU}(\wt{\bfs}(\bfi))\simeq \BC[S]$, and $\overline{\CA}(\wt{\bfs}(\bfi))=\overline{\CU}(\wt{\bfs}(\bfi))$.
\end{cor}

\begin{proof}
For each $i\in I$, the frozen variable $\Delta_{\omega_i,w_0\omega_i}$ in $\wt{\bfs}(\bfi)$ generates the ideal for the irreducible closed subset $\overline{\mathring{S}_{w_0s_i}}=S\cap \overline{B^+w_0s_iB^+}$ in $S$. Therefore, the frozen variables are prime in $\BC[S]$. 

The second claim $\overline{\CA}=\overline{\CU}$ follows from the fact that the matrix entries are cluster variables and generate $\BC[S]$.
\end{proof}

\subsubsection{The symmetric matrices $\Sym_n$}
\label{sec:partialcpt:appl:Symn}

Let $\Sym_n\subset M_n$ be the variety of symmetric $n\x n$ matrices over $\BC$. Cluster structures on $M_n$ were constructed in \citelist{\cite{FWZ21b}\cite{SY25}}. Here we translate the construction in \cite{FWZ21b} in terms of double Bruhat cells in $SL_{n+1}$.

Embed $GL_n$ in $SL_{n+1}$ by sending $g$ to the block diagonal matrix $\begin{pmatrix}
g & 0\\
0 & \det(g)^{-1}
\end{pmatrix}$. Then the Weyl group $W(GL_n)$ is identified as a parabolic subgroup of $W(SL_{n+1})$ with indices $1,2,\cdots,n-1$. For any $u,v\in W(GL_n)$, the double Bruhat cell $GL_n^{u,v}$ is mapped isomorphically onto $SL_{n+1}^{u,v}$. Then for the longest element $w_1$ of $W(GL_n)$, the cluster structure $\CU$ on $SL_{n+1}^{w_1,w_1}$ induces a cluster structure $\BC[M_n]\simeq \overline{\CU}$ by partial compactification. The cluster variables for $M_n$ are exactly the corresponding matrix minors from the cluster variables for $SL_{n+1}^{w_1,w_1}$. By construction, these minors never involve the last row and column in $SL_{n+1}$.

\begin{cor}
\label{cor:Symn}
For any $u,v\in W(GL_n)$ such that $v^{-1}u=w_1$ with $\ell(w_1)=\ell(u)+\ell(v)$, and double reduced word $\bfi$ for $(u,v)$, we have an isomorphism $\overline{\CU}(\wt{\bfs}(\bfi))\simeq \BC[\Sym_n]$, and $\overline{\CA}(\wt{\bfs}(\bfi))=\overline{\CU}(\wt{\bfs}(\bfi))$.
\end{cor}

\begin{proof}
We have $\CU(\wt{\bfs}(\bfi))\simeq \BC[\Sym_{n+1}\cap SL_{n+1}^{w_1,w_1}]\simeq \BC[\Sym_n\cap GL_n^{w_1,w_1}]$.

$\BC[\Sym_n]$ is a polynomial ring hence a UFD. The frozen variables are the determinants of the top-right $k\x k$ submatrices for $k\leq n$. We show that they are irreducible hence prime on $\Sym_n$ by induction. The base cases $n=1,2$ are trivial.

Let $P=\{g_{ij}\}_{i\leq j}$ be the set of coordinates for $\Sym_n$, so $\BC[\Sym_n]$ is isomorphic to the polynomial ring generated by $P$. By cofactor expansion, the determinant can be written as $ag_{11}+b$, where $a$ is the determinant for $\Sym_{n-1}$ obtained by removing the first row and column, and $b\in \BC[P\backslash\{g_{11}\}]$. Then $a$ is irreducible by induction hypothesis. By a degree count for the coordinates in $\Sym_{n-1}$, we see that $a$ does not divide $b$, so the determinant is irreducible.

Any other frozen variable in $\wt{\bfs}(\bfi)$ can be written as $cg_{1n}+d$, where $c$ is a frozen variable for $\Sym_{n-2}$ obtained by removing the first and last rows and columns, and $d\in \BC[P\backslash\{g_{1n}\}]$. The same argument shows that it is irreducible, and we conclude that $\overline{\CU}(\wt{\bfs}(\bfi))\simeq \BC[\Sym_n]$ by Proposition~\ref{prop:CptFold}.

$\overline{\CA}=\overline{\CU}$ again follows from the fact that the matrix entries are cluster variables and generate $\BC[\Sym_n]$.
\end{proof}

\subsection{An exceptional isomorphism}
\label{sec:partialcpt:iso}

We also record an exceptional isomorphism discovered from observing the cluster structure on $\Sym_n$.

Let $J$ be the $n\x n$ matrix with entries $c_{ij}=1$ if $i+j=n+1$ and $c_{ij}=0$ otherwise. Consider the symplectic group $G=Sp_{2n}\subset GL_{2n}$ defined using the symplectic form $\Omega=\begin{pmatrix}
O & J\\
-J & O
\end{pmatrix}$ and $B^+$ and $B^-$ be the Borel subgroup of upper and lower triangular matrices in $G$ respectively. Let $\CB=G/B^+$ be the flag variety. For any $w\in W=C_n$, we have the Schubert cell $\mathring{\CB}_w=B^+wB^+/B^+$, which is isomorphic to $\BC^{\ell(w)}$.

Let $\Phi$ be the map $\Sym_n\to \mathring{\CB}_{w^J}$ given by $X\mapsto \begin{pmatrix}
I & JX\\
0 & I
\end{pmatrix}\dot{w}^JB^+$, where $w^J$ is the minimal length representative of the longest element in the quotient $C_n/A_{n-1}$. Explicitly, $w^J$ has entries $c_{ij}=(-1)^i$ if $i\neq j$ and $i\equiv j\pmod{n}$, and $c_{ij}=0$ otherwise. Since $\mathring{\CB}_{w^J}\simeq U^+\cap w^JU^-{w^J}^{-1}$, a direct calculation shows that $\Phi$ is an isomorphism. Furthermore, $\Phi(X)\in B^-B^+/B^+$ if and only if $\begin{pmatrix}
I & JX\\
0 & I
\end{pmatrix}\dot{w}^J\in B^-B^+$. By checking on the principal minors of the product, this is equivalent to $X$ being in the open double Bruhat cell in $GL_n$.

Let
\[
\bfi=(-1,-2,\cdots,-(n-1),-1,-2,\cdots,-(n-2),\cdots,-1),
\]
so we get a seed $\bfs=\wt{\bfs}(\bfi)$ for $\Sym_n$. Its vertices are indexed by $(i,k)$ for $1\leq i\leq n$ and $0\leq k\leq n-i$, where the vertices $(i,0)$ are frozen. The cluster variables are of the form $\Delta_{\omega_i,w\omega_i}$.

Meanwhile, the open Richardson variety $\mathring{\CB}_{e,w^J}=B^-B^+/B^+\cap B^+w^JB^+/B^+$ is a quotient of $G^{w^J,e}$ by the maximal torus $T$, so it inherits a cluster structure $G^{w^J,e}$ by deleting a set of frozen variables. $w^J$ has reduced word
\[
\bfi_C=(n,n-1,\cdots,1,n,n-1,\cdots,2,\cdots,n),
\]
where $n$ corresponds to the long root of $C_n$. Then we also get a seed $\bfs_C=\bfs(-\bfi_C)$ for $\mathring{\CB}_{e,w^J}$ with vertices indexed by $(i,k)$, $1\leq i \leq n$, $1\leq k\leq i$, where the vertices $(i,i)$ are frozen. The cluster variables are of the form $\Delta_{\omega_i,w\omega_i}(z)$ for $z\dot{w}^JB^+\in \mathring{\CB}_{e,w^J}$ where $z\in U^+$. Here the frozen vertices $(i,0)$ for the cluster structure on $G^{w^J,e}$ are deleted. By \cite{GLS11}*{Theorem~3.3}, the partial compactification $\overline{\CU}(\bfs(\bfi_C))\simeq \BC[\mathring{\CB}_{w^J}]$.

By construction, the vertices in $\wt{\bfs}(\bfi)$ with $k=1$ have multiplicity 1 and the others have multiplicity $\frac{1}{2}$. For $\bfs(\bfi_C)$, the vertices with $i=1$ have multiplicity 1 while the others have multiplicity $\frac{1}{2}$.

\begin{prop}
\label{prop:Exceptional}
The isomorphism $\Phi:\Sym_n\isoto \mathring{\CB}_{w^J}$ respects the cluster structures, where the index $(i,k)$ in $\bfs'=\bfs(\bfi_C)$ is identified with the index $(n+1-k,k-i)$ in $\bfs=\wt{\bfs}(\bfi)$.
\end{prop}

\begin{proof}
The exchange matrices match by construction of $\bfs$ and $\bfs_C$, where their quivers are orientations of a triangular tiling, with weighted arrows.

We now compute $A_{(i,k)}'(\Phi(X))$ in terms of $X$. By definition, they are of the form $\Delta_{\omega_i,w\omega_i}\begin{pmatrix}
I & JX\\
0 & I
\end{pmatrix}$ for the generalized minors $\Delta_{\omega_i,w\omega_i}$ in $Sp_{2n}$. The principal minor $\Delta_{\omega_i,\omega_i}$ is the determinant of the top-left $i\x i$ submatrix, so $A_{(i,k)}'(\Phi(X))$ is given by the corresponding minor of $\begin{pmatrix}
I & JX\\
0 & I
\end{pmatrix}\dot{u}_{\leq(i,k)}$. It is easy to check that this is equal to the minor of $X$ with rows $\{k-i+1,\cdots,n+1-i\}$ and columns $\{1,\cdots,n+1-k\}$, so we have $\Phi^*(A_{(i,k)}')=A_{(n+1-k,k-i)}$. This shows that the cluster variables match, so we have a cluster isomorphism.
\end{proof}

The subvariety $\Sym_n\cap GL_n$ is stratified by $\mathring{S}_w$ ($w\in A_{n-1}$) from the Bruhat decomposition of $GL_n$, while $\mathring{\CB}_{w^J}$ is naturally stratified by the open Richardson varieties $\mathring{\CB}_{v,w^J}$ for $v\leq w^J$ in $C_n$.

\begin{prop}
\label{prop:Phi_stratification}
The isomorphism $\Phi$ respects the stratifications, mapping $\mathring{S}_w$ to $\mathring{\CB}_{w_Jw,w^J}$, where $w_J$ is the longest element of $A_{n-1}$. The singular matrices in $\Sym_n$ are mapped to the closure of $\mathring{\CB}_{s_n,w^J}$ in $\mathring{\CB}_{w^J}$.
\end{prop}

\begin{proof}
The closure of $\mathring{S}_w$ in $\Sym_n\cap GL_n$ is defined by $\Delta_{u\omega_i,\omega_i}(X)$ for all $u\omega_i\not\leq w\omega_i$. Meanwhile, the closure of $\mathring{\CB}_{v,w^J}$ in $\mathring{\CB}_{w^J}$ is defined by $\Delta_{u\omega_i,\omega_i}(gB^+)$ for all $u\omega_i\not\geq v\omega_i$.

For $X\in \Sym_n$, $\det X=0$ if and only if $\Delta_{\omega_n,\omega_n}(\Phi(X))=0$. This equation defines the closure of $\mathring{\CB}_{s_n,w^J}$, so we get the claim for the singular matrices.

The complement of the closure of $\mathring{\CB}_{s_n,w^J}$ is the union of $\mathring{\CB}_{v,w^J}$ for $v\in A_{n-1}$. Therefore, for any $X\in \Sym_n\cap GL_n$, we have $\Phi(X)\in \mathring{\CB}_{v,w^J}$ for some $v\in A_{n-1}$. To determine $v$, we only need to check $\Delta_{u\omega_i,\omega_i}(\Phi(X))$ for $u\in A_{n-1}$ and $i\neq n$. Since $J=\dot{w}_J$ in $GL_n$ up to $\pm$ signs in the entries, while $\dot{w}^J=\begin{pmatrix}
O & I\\
-I & O
\end{pmatrix}$ in $Sp_{2n}$ up to $\pm$ signs in the entries, we have 
\[
\Delta_{w_Ju\omega_i,\omega_i}\left(\begin{pmatrix}
I & JX\\
0 & I
\end{pmatrix}\dot{w}^J\right)=\pm \Delta_{u\omega_i,\omega_i}(X),\]
and the result follows.
\end{proof}

\begin{example}
We continue with our running example. The quiver for $\Sym_3$ is obtained from the quiver for $\mathring{S}_{w_0}$ in $SL_3/SO_3$ by adding a frozen vertex for the determinant.
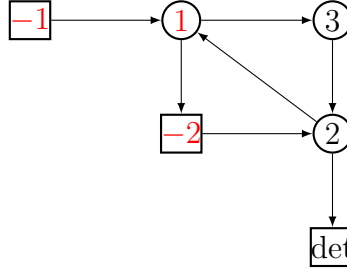
\begin{center}
\begin{tikzpicture}[every node/.style={inner sep=0, minimum size=0.5cm, thick, fill=white, draw}, x=2cm, y=1cm]
\node (0) at (0,0) {\red{$-1$}};
\node (1) at (1,0) [circle]{\red{$1$}};
\node (2) at (2,0) [circle]{$3$};
\node (5) at (1,-1.5) {\red{$-2$}};
\node (6) at (2,-1.5) [circle]{$2$};
\node (7) at (2,-3) {$\det$};
\drawpath{0,1,2,6,1,5,6,7}{black}
\end{tikzpicture}
\end{center}
Rotating the quiver by 90 degrees, we get a quiver for $\mathring{\CB}_{w^J}$ in $Sp_6$, where $w^J=s_3s_2s_1s_3s_2s_3$.
\end{example}


\begin{thebibliography}{}

\bib{Bao26}{article}{
   title={Total positive and symmetric spaces}, 
   author={Bao, H.},
   year={2026},
   eprint={https://arxiv.org/abs/2606.25516},
}

\bib{BFZ05}{article}{
   author={Berenstein, A.},
   author={Fomin, S.},
   author={Zelevinsky, A.},
   title={Cluster algebras III: Upper bounds and double Bruhat cells},
   journal={Duke Math. J.},
   volume={126},
   number={1},
   date={2005},
   pages={1--52},
}

\bib{BZ05}{article}{
   author={Berenstein, A.},
   author={Zelevinsky, A.},
   title={Quantum cluster algebras},
   journal={Adv. Math.},
   volume={195},
   date={2005},
   number={2},
   pages={405--455},
}

\bib{Dem08}{article}{
   author={Demonet, L.},
   title={Alg\`ebres amass\'ees et alg\`ebres pr\'eprojectives: le cas non
   simplement lac\'e},
   language={French, with English and French summaries},
   journal={C. R. Math. Acad. Sci. Paris},
   volume={346},
   date={2008},
   number={7-8},
   pages={379--384},
}

\bib{Dup08}{article}{
   author={Dupont, G.},
   title={An approach to non-simply laced cluster algebras},
   journal={J. Algebra},
   volume={320},
   date={2008},
   number={4},
   pages={1626--1661},
}

\bib{EE05}{article}{
   author={Enriquez, B.},
   author={Etingof, P.},
   title={Quantization of classical dynamical $r$-matrices with nonabelian base},
   journal={Comm. Math. Phys.},
   volume={254},
   date={2005},
   number={3},
   pages={603--650},
}

\bib{EL07}{article}{
   author={Evens, S.},
   author={Lu, J.-H.},
   title={Poisson geometry of the Grothendieck resolution of a complex semisimple group},
   language={English, with English and Russian summaries},
   journal={Mosc. Math. J.},
   volume={7},
   date={2007},
   number={4},
   pages={613--642, 766},
}

\bib{FG06}{article}{
   author={Fock, V. V.},
   author={Goncharov, A. B.},
   title={Cluster $\scr X$-varieties, amalgamation, and Poisson-Lie groups},
   conference={
      title={Algebraic geometry and number theory},
   },
   book={
      series={Progr. Math.},
      volume={253},
      publisher={Birkh\"auser Boston, Boston, MA},
   },
   date={2006},
   pages={27--68},
}

\bib{FST12}{article}{
   author={Felikson, A.},
   author={Shapiro, M.},
   author={Tumarkin, P.},
   title={Cluster algebras of finite mutation type via unfoldings},
   journal={Int. Math. Res. Not. IMRN},
   date={2012},
   number={8},
   pages={1768--1804},
}

\bib{FWZ21a}{article}{
   title={Introduction to Cluster Algebras. Chapters 4-5}, 
   author={Fomin, S.},
   author={Williams, L.},
   author={Zelevinsky, A.},
   year={2021},
   eprint={https://arxiv.org/abs/1707.07190},
}

\bib{FWZ21b}{article}{
   title={Introduction to Cluster Algebras. Chapter 6}, 
   author={Fomin, S.},
   author={Williams, L.},
   author={Zelevinsky, A.},
   year={2021},
   eprint={https://arxiv.org/abs/2008.09189},
}

\bib{FZ99}{article}{
   author={Fomin, S.},
   author={Zelevinsky, A.},
   title={Double Bruhat cells and total positivity},
   journal={J. Amer. Math. Soc.},
   volume={12},
   date={1999},
   number={2},
   pages={335--380},
}

\bib{FZ02}{article}{
   author={Fomin, S.},
   author={Zelevinsky, A.},
   title={Cluster algebras. I. Foundations},
   journal={J. Amer. Math. Soc.},
   volume={15},
   date={2002},
   number={2},
   pages={497--529},
}

\bib{FZ03}{article}{
   author={Fomin, S.},
   author={Zelevinsky, A.},
   title={Cluster algebras. II. Finite type classification},
   journal={Invent. Math.},
   volume={154},
   date={2003},
   number={1},
   pages={63--121},
}

\bib{Fra16}{article}{
   author={Fraser, C.},
   title={Quasi-homomorphisms of cluster algebras},
   journal={Adv. in Appl. Math.},
   volume={81},
   date={2016},
   pages={40--77},
}

\bib{GLS26}{article}{
   author={Galashin, P.},
   author={Lam, T.},
   author={Sherman-Bennett, M.},
   title={Braid variety cluster structures, II: general type},
   journal={Invent. Math.},
   volume={243},
   date={2026},
   number={3},
   pages={1079--1127},
}

\bib{GP93}{article}{
   author={Geck, M.},
   author={Pfeiffer, G.},
   title={On the irreducible characters of Hecke algebras},
   journal={Adv. Math.},
   volume={102},
   date={1993},
   number={1},
   pages={79--94},
}

\bib{GLS11}{article}{
   author={Gei\ss, C.},
   author={Leclerc, B.},
   author={Schr\"oer, J.},
   title={Kac-Moody groups and cluster algebras},
   journal={Adv. Math.},
   volume={228},
   date={2011},
   number={1},
   pages={329--433},
}

\bib{GSV10}{book}{
   author={Gekhtman, M.},
   author={Shapiro, M.},
   author={Vainshtein, A.},
   title={Cluster algebras and Poisson geometry},
   series={Mathematical Surveys and Monographs},
   volume={167},
   publisher={American Mathematical Society, Providence, RI},
   date={2010},
   pages={xvi+246},
}

\bib{GHKK18}{article}{
   author={Gross, M.},
   author={Hacking, P.},
   author={Keel, S.},
   author={Kontsevich, M.},
   title={Canonical bases for cluster algebras},
   journal={J. Amer. Math. Soc.},
   volume={31},
   date={2018},
   number={2},
   pages={497--608},
}

\bib{IM24}{article}{
   author={Ip, I. C.-H.},
   author={Man, R.},
   title={Positive representations with zero Casimirs},
   journal={Comm. Math. Phys.},
   volume={405},
   date={2024},
   number={4},
   pages={Paper No. 91, 60},
}

\bib{Kac83}{book}{
   author={Kac, V. G.},
   title={Infinite-dimensional Lie algebras},
   series={Progress in Mathematics},
   volume={44},
   note={An introduction},
   publisher={Birkh\"auser Boston, Inc., Boston, MA},
   date={1983},
   pages={xvi+245},
   isbn={0-8176-3118-6},
   doi={10.1007/978-1-4757-1382-4},
}

\bib{Lu14}{article}{
   author={Lu, J.-H.},
   title={On the $T$-leaves and the ranks of a Poisson structure on twisted conjugacy classes},
   journal={Indag. Math. (N.S.)},
   volume={25},
   date={2014},
   number={5},
   pages={1102--1121},
}

\bib{Lus93}{book}{
   author={Lusztig, G.},
   title={Introduction to quantum groups},
   series={Progress in Mathematics},
   volume={110},
   publisher={Birkh\"{a}user Boston, Inc., Boston, MA},
   date={1993},
   pages={xii+341},
}

\bib{Mul13}{article}{
   author={Muller, G.},
   title={Locally acyclic cluster algebras},
   journal={Adv. Math.},
   volume={233},
   date={2013},
   pages={207--247},
}

\bib{Oya25}{article}{
   author={Oya, H.},
   title={A note on a cluster structure of the coordinate ring of a simple algebraic group},
   year={2025},
   eprint={https://arxiv.org/abs/2504.09011}, 
}

\bib{Qin25}{article}{
   title={Analogs of the dual canonical bases for cluster algebras from Lie theory}, 
   author={Qin, F.},
   year={2025},
   eprint={https://arxiv.org/abs/2407.02480},
}

\bib{QY25}{article}{
    author={Qin, F.},
    author={Yakimov, M.},
    title={Partially compactified quantum cluster structures on simple algebraic groups and the full Berenstein--Zelevinsky conjecture}, 
    year={2025},
    eprint={https://arxiv.org/abs/2504.05134},
}

\bib{SW21}{article}{
   author={Shen, L.},
   author={Weng, D.},
   title={Cluster structures on double Bott--Samelson cells},
   journal={Forum Math. Sigma},
   volume={9},
   date={2021},
   pages={Paper No. e66, 89},
}

\bib{Spr85}{article}{
   author={Springer, T. A.},
   title={Some results on algebraic groups with involutions},
   conference={
      title={Algebraic groups and related topics},
      address={Kyoto/Nagoya},
      date={1983},
   },
   book={
      series={Adv. Stud. Pure Math.},
      volume={6},
      publisher={North-Holland, Amsterdam},
   },
   date={1985},
   pages={525--543},
}

\bib{STW16}{article}{
   author={Sturmfels, B.},
   author={Tsukerman, E.},
   author={Williams, L.},
   title={Symmetric matrices, Catalan paths, and correlations},
   journal={J. Combin. Theory Ser. A},
   volume={144},
   date={2016},
   pages={496--510},
}

\bib{SY25}{article}{
   title={Reductive monoids and cluster algebras}, 
   author={Song, J.},
   author={Ye, J. Y.},
   year={2025},
   eprint={https://arxiv.org/abs/2512.09564},
}

\bib{Zho20}{article}{
   author={Zhou, Y.},
   title={Cluster structures and subfans in scattering diagrams},
   journal={SIGMA Symmetry Integrability Geom. Methods Appl.},
   volume={16},
   date={2020},
   pages={Paper No. 013, 35},
}

\end{thebibliography}
\end{document}